\documentclass[12pt]{amsart}
\usepackage{}

\usepackage{amsmath}
\usepackage{amsfonts}
\usepackage{amssymb}
\usepackage[all]{xy}           %xypic macro for latex2.09

\usepackage{bbding}
\usepackage{txfonts}
\usepackage{amscd}

\usepackage[shortlabels]{enumitem}
\usepackage{ifpdf}
\ifpdf
  \usepackage[colorlinks,final,backref=page,hyperindex]{hyperref}
\else
  \usepackage[colorlinks,final,backref=page,hyperindex,hypertex]{hyperref}
\fi
\usepackage{tikz}
\usepackage[active]{srcltx}

\makeatletter

\newtheorem{thm}{Theorem}[section]
\newtheorem{lem}[thm]{Lemma}
\newtheorem{cor}[thm]{Corollary}
\newtheorem{pro}[thm]{Proposition}
\newtheorem{ex}[thm]{Example}
\newtheorem{rmk}[thm]{Remark}
\newtheorem{defi}[thm]{Definition}

\newcommand {\emptycomment}[1]{}

\newcommand{\lon }{\,\rightarrow\,}
\newcommand{\be }{\begin{equation}}
\newcommand{\ee }{\end{equation}}

\newcommand{\g}{\mathfrak g}

\newcommand{\huaA}{\mathcal{A}}%{{\mathcal{F}}}%{\mathcal{A}}

\newcommand{\huaG}{\mathcal{G}}

\newcommand{\huaO}{{\mathcal{O}}}

\newcommand{\Courant}[1]{\left\llbracket  #1\right\rrbracket }

\newcommand{\Id}{\rm{Id}}

\newcommand{\br}[1]{   [ \cdot,    \cdot  ]   }

\newcommand{\dM}{\mathrm{d}}

\newcommand{\Hom}{\mathrm{Hom}}

\newcommand{\gl}{\mathfrak {gl}}

\begin{document}

\title[deformation map]{Deformation maps of Quasi-twilled associative algebras
}

\author{Shanshan Liu}
\address{College of Big Data and Internet, Shenzhen Technology University, Shenzhen, 518118, China}
\email{liushanshan@sztu.edu.cn}

\author{Abdenacer Makhlouf}
\address{University of Haute Alsace, IRIMAS- D\'epartement  de Math\'ematiques, Universit\'e de Haute Alsace, France}
\email{abdenacer.makhlouf@uha.fr}

\author{Lina Song}
\address{Department of Mathematics, Jilin University, Changchun 130012, Jilin, China}
\email{songln@jlu.edu.cn}

%\date{\today}

\begin{abstract}
In this paper, we introduce two types of deformation maps of quasi-twilled associative algebras. Each type of deformation maps unify various operators on associative algebras.  Right deformation maps unify
  modified Rota-Baxter operators of weight $\lambda$,   derivations, homomorphisms and crossed homomorphisms. Left deformation maps unify relative Rota-Baxter operators of weight 0,  twisted Rota-Baxter operators,   Reynolds operators   and   deformation maps of  matched pairs of associative algebras. Furthermore, we give the controlling algebra and the cohomology of these two types of deformation maps. On the one hand, we obtain some existing results for modified Rota-Baxter operators of weight $\lambda$, derivations, homomorphisms, crossed homomorphisms,   relative Rota-Baxter operators of weight 0,   twisted Rota-Baxter operators and Reynolds operators. On the other hand, we also obtain some new results, such as the controlling algebra of a modified Rota-Baxter operator of weight $\lambda$ on an associative algebra, the controlling algebra and the cohomology of a deformation map of a matched pair of associative algebras.
\end{abstract}
%\subjclass[2010]{17B10, 17B56, 17A42}
%\footnotetext{{\it{MSC}}: 16T25, 17B62, 17B99.}

\maketitle

\tableofcontents

\allowdisplaybreaks

%\end{document}

\section{Introduction}
The concept of a formal deformation theory of an algebraic structure began with the seminal work of Gerstenhaber for associative algebras~\cite{Gerstenhaber2}. Subsequently, the theory has been generalized by Nijenhuis and Richardson for Lie algebras~\cite{NR,NR2}. It has been known from Gerstenhaber that the deformation of some algebraic structure is governed by a suitable cohomology theory of the structure. For instance, the deformation of an associative algebra is governed by the classical Hochschild cohomology of the associative algebra \cite{Gerstenhaber1,Hochschild}, while the deformation of a Lie algebra is governed by the Chevalley-Eilenberg cohomology \cite{NR}. More generally, deformation theory for algebras over quadratic operads was developed by Balavoine~\cite{Bal}. For more general operads, we can see the reference \cite{KSo,LV,Ma}.

The cohomology and deformation theories of various operators on associative algebras were established with fruitful applications, such as derivations  \cite{Guo}, crossed homomorphisms \cite{Das3}, associative algebra homomorphisms \cite{Barmeier,BD,Fregier,Gerstenhaber3}, relative Rota-Baxter operators of weight 0 (also called $\huaO$-operators) \cite{Bai1,Bai2,Bai3,Das2,Das4,Uchino},  twisted Rota-Baxter operators and  Reynolds operators  \cite{Das}. By using the method of derived brackets \cite{Kosmann-Schwarzbach,Ma-0,Vo}, one can construct the controlling algebra, namely an algebra whose Maurer-Cartan elements are the given structure. This is the most important step in the above studies. Then twisting the controlling algebra by a Maurer-Cartan element, one can obtain the algebra that governs deformations of the given operator, as well as the coboundary operator in the deformation complex. A solution of the associative Yang-Baxter equation of
weight $\lambda \in \mathbb{K}$ (in short modified AYBE of weight $\lambda$) \cite{Guo1,Guo2} is called a modified Rota-Baxter operator of weight $\lambda$. In \cite{Das1}, the author give the cohomology of modified Rota-Baxter operators of weight $\lambda$ on associative algebras. In \cite{Agore}, the notion of a deformation map of a matched pair of associative algebras was introduced in the study of classifying compliments.

The purpose of the paper is to unify all aforementioned operators. We introduce two types of deformation maps of quasi-twilled associative algebras. Right deformation maps unify
a modified Rota-Baxter operator of weight $\lambda$ on an associative algebra, a derivation on an associative algebra, a crossed homomorphism between associative algebras and an associative algebra homomorphism. Left deformation maps unify a relative Rota-Baxter operator of weight 0 on an associative algebra, a twisted Rota-Baxter operator on an associative algebra, a Reynolds operator on an associative algebra and a deformation map of a matched pair of associative algebras. We give the controlling algebra and the cohomology of these two types of deformation maps and obtain some existing results. We also obtain some new results. We give the controlling algebra of a modified Rota-Baxter operator of weight $\lambda$ on an associative algebra. We also give the controlling algebra and the cohomology of a deformation map of a matched pair of associative algebras.

The paper is organized as follows. In Section \ref{sec:quasi-twilled associatve algebra}, we recall the notion of a quasi-twilled associative algebra and give some examples of it. In Section \ref{sec:deformation map-1}, we introduce the notion of a right deformation map of a quasi-twilled associative algebra and give some examples of it, such as a modified Rota-Baxter operator of weight $\lambda$ on an associative algebra, a derivation on an associative algebra, a crossed homomorphism between associative algebras and an associative algebra homomorphism.  We give the controlling algebra (curved $L_\infty$-algebra) and the cohomology of a right deformation map of a quasi-twilled associative algebra. As a byproduct, we give the controlling algebra of a modified Rota-Baxter operator of weight $\lambda$ on an associative algebra. In Section \ref{sec:deformation map-2},  we introduce the notion of a left deformation map of a quasi-twilled associative algebra and give some examples of it, such as a relative Rota-Baxter operator of weight 0 on an associative algebra, a twisted Rota-Baxter operator on an associative algebra, a Reynolds operator on an associative algebra and a deformation map of a matched pair of associative algebras. We give the controlling algebra (curved $L_\infty$-algebra) and the cohomology of a left deformation map of a quasi-twilled associative algebra. As a byproduct, we give the controlling algebra and the cohomology of a deformation map of a matched pair of associative algebras.
%\vspace{2mm}
%\noindent
%{\bf Acknowledgements. }  We give warmest thanks to Yunhe Sheng   for helpful comments that improve the paper.   This work is supported by
 %National Natural Science Foundation of China (Grant No. 12001226).
\section{Quasi-twilled associative algebras}\label{sec:quasi-twilled associatve algebra}
In this section, we recall the notion of a quasi-twilled associative algebra and give some examples.
%\begin{defi}
%An {\bf associative algebra} $(\g,\cdot_{\g})$ is a vector space $\g$ equipped with a bilinear product $\cdot_{\g}:\g\otimes \g\longrightarrow \g$, such that for all $x,y,z\in \g$, the following equality is satisfied:
%\begin{equation*}
%(x\cdot_{\g} y)\cdot_{\g} z=x\cdot_{\g} (y\cdot_{\g} z).
%\end{equation*}
%\end{defi}
\begin{defi}{\rm(\cite{LJF-1})}
 A {\bf representation} of an associative algebra $(A,\cdot_{A})$ on a vector space $V$ consists of a pair $(\rho,\mu)$, where $\rho,\mu:A\longrightarrow \gl(V)$ are linear maps, such that for all $x,y\in A$:
\begin{equation*}
\rho(x\cdot_{A} y)=\rho(x)\circ\rho(y),\quad \mu(x \cdot_{A} y)=\mu(y)\circ\mu(x),\quad \rho(x)\circ\mu(y)=\mu(y)\circ\rho(x).
\end{equation*}
\end{defi}
We denote a representation of an associative algebra $(A,\cdot_{A})$ by $(V;\rho,\mu)$. Furthermore, let $L,R:A\longrightarrow \gl(A)$ be linear maps, where $L_xy=x
\cdot_{A} y, R_xy=y\cdot_{A} x$. Then $(A;L,R)$ is also a representation, which is called the regular representation.
\begin{defi}{\rm(\cite{Das3})}
 An {\bf associative representation} of an associative algebra $(A,\cdot_{A})$ on an associative algebra $(A',\cdot_{A'})$ consists of a pair $(\rho,\mu)$, where $(A';\rho,\mu)$ is a representation of $(A,\cdot_{A})$ on $(A',\cdot_{A'})$, such that for all $x\in A, u,v\in A'$:
\begin{equation*}
\rho(x)(u\cdot_{A'} v)=\rho(x)(u)\cdot_{A'} v,\quad u\cdot_{A'}\rho(x)(v)=\mu(x)(u)\cdot_{A'} v,\quad u\cdot_{A'}\mu(x)(v)=\mu(x)(u\cdot_{A'} v).
\end{equation*}
\end{defi}

Let $A$ be a vector space. We consider the graded vector space $\oplus_{n=0}^{+\infty}\Hom(\otimes^{n+1}A, A)$ with the degree of elements in $\Hom(\otimes^{n+1}A, A)$ being $n$.
\begin{thm}{\rm(\cite{Gerstenhaber1,Gerstenhaber2})}
The graded vector space $\oplus_{n=0}^{+\infty}\Hom(\otimes^{n+1}A, A)$ equipped with the Gerstanhaber bracket
\begin{equation*}
[f,g]_G=f\circ g-(-1)^{(m-1)(n-1)}g\circ f, \quad
\forall~f\in \Hom(\otimes^mA, A), g \in \Hom(\otimes^nA, A),
\end{equation*}
 is a graded Lie algebra, where $f\circ g \in \Hom(\otimes^{m+n-1}A, A)$ is defined by
\begin{equation}
f\circ g(x_1,\dots,x_{m+n-1})=\sum_{i=1}^m(-1)^{(i-1)(n-1)} f(x_1,\dots,x_{i-1},g(x_i,\dots,x_{i+n-1}),x_{i+n},\dots,x_{m+n-1}),
\end{equation}
for all $x_1,\dots,x_{m+n-1} \in A.$

Furthermore, $(A,\pi)$ is an associative algebra, where $\pi \in \Hom(\otimes^2A, A)$, if and only if $\pi$ is a Maurer-Cartan element of the graded Lie algebra $(\oplus_{n=0}^{+\infty}\Hom(\otimes^{n+1}A, A),[\cdot,\cdot]_G)$, i.e. $[\pi,\pi]_G=0.$
\end{thm}
For all $f\in \Hom(A\otimes A, A), g \in \Hom(A\otimes A, A)$, $x_1, x_2,x_3 \in A$, denote $f\circ g$ by
\begin{equation*}
f\circ g=(f\circ g)_1-(f\circ g)_2,
\end{equation*}
where $(f\circ g)_1(x_1,x_2,x_3)=f(g(x_1,x_2),x_3)$ and $(f\circ g)_2(x_1,x_2,x_3)=f(x_1,g(x_2,x_3))$.
%If $\pi$ is a Maurer-Cartan element of the graded Lie algebra $(\oplus_{n=0}^{+\infty}\Hom(\otimes^{n+1}\g, \g),[\cdot,\cdot]_G)$, then $d_\pi(f):=[\pi,f]_G$ is a graded derivation of $(\oplus_{n=0}^{+\infty}\Hom(\otimes^{n+1}\g, \g),[\cdot,\cdot]_G)$ satisfying $d_\pi\circ d_\pi=0$. So that $(\oplus_{n=0}^{+\infty}\Hom(\otimes^{n+1}\g, \g),[\cdot,\cdot]_G,d_\pi)$ becomes a differential graded Lie algebra.

Let $A_1$ and $A_2$ be vector spaces. The elements in $A_1$ are denote by $x$ and the elements in $A_2$ are denoted by $u$. Let $c:\otimes^nA_2\longrightarrow A_1$ be a linear map. We can construct a cochain $\hat{c}\in \Hom(\otimes^n(A_1\oplus A_2),A_1\oplus A_2)$ by
\begin{equation}
\hat{c}((x_1,u_1),\dots,(x_n,u_n)):=(c(u_1,\dots,u_n),0).
\end{equation}
In general, for a given multilinear map $f:A_{i(1)}\otimes A_{i(2)}\otimes\dots \otimes A_{i(n)}\longrightarrow A_j, i(1),\dots, i(n), j\in \{1,2\}$, we define a cochain $\hat{f}\in \Hom(\otimes^n(A_1\oplus A_2),A_1\oplus A_2)$ by
\begin{equation}\hat{f}:=\left\{\begin{array}{ll}
 &f\quad \text{on}~A_{i(1)}\otimes A_{i(2)}\otimes\dots \otimes A_{i(n)},\\
&0\quad \text{all other cases}.
\end{array}\right.
\end{equation}
We call $\hat{f}$ a {\bf lift} of $f$. We denote by $A^{k,l}$ the direct sum of all $(k+l)$-tensor powers of $A_1$ and $A_2$, where $k$ (resp. $l$) is the number of $A_1$ (resp. $A_2$). The tensor space $\otimes^n(A_1\oplus A_2)$ is expanded into the direct sum of $A^{k,l},k+l=n$. Then we have
\begin{equation*}
\Hom(\otimes^n(A_1\oplus A_2),A_1\oplus A_2)\cong (\oplus_{k+l=n}\Hom(A^{k,l},A_1))\oplus (\oplus_{k+l=n}\Hom(A^{k,l},A_2)),
\end{equation*}
where the isomorphism is given by the lift. See \cite{Uchino} for more details. Then $((\oplus_{k+l=n}\Hom(A^{k,l},A_1))\oplus (\oplus_{k+l=n}\Hom(A^{k,l},A_2)),[\cdot,\cdot]_G)$ is a graded Lie algebra, where $[\cdot,\cdot]_G$ is given by
\begin{equation}
[\hat{f},\hat{g}]_G=\widehat{[f,g]}_G.
\end{equation}
In this paper, we omit the notion of lift of a map without ambiguity.
\begin{defi}{\rm(\cite{Uchino})}
Let $(\huaA,\cdot_{\huaA})$ be an associative algebra with a decomposition into two subspaces $\huaA=A\oplus A'$. The triple $(\huaA,A,A')$ is called a {\bf quasi-twilled associative algebra} if $A'$ is a subalgebra of $(\huaA,\cdot_{\huaA})$.
\end{defi}
Let $(\huaA,\cdot_{\huaA})$ be a quasi-twilled associative algebra. Denote the multiplication on $\huaA$ by $\Omega$. Then there exists $\pi \in \Hom(A\otimes A,A)$, $\xi \in \Hom(A\otimes A',A)$, $\eta \in \Hom(A'\otimes A,A)$, $\beta \in \Hom(A'\otimes A',A')$, $\rho \in \Hom(A\otimes A',A')$, $\mu \in \Hom(A'\otimes A,A')$,  $\theta \in \Hom(A\otimes A,A')$, such that
\begin{equation*}
\Omega=\pi+\xi+\eta+\beta+\rho+\mu+\theta.
\end{equation*}
More precisely, for all $x,y \in A, u,v\in A'$, we have
\begin{equation}
\Omega((x,u),(y,v))=(\pi(x,y)+\xi(x,v)+\eta(u,y),\beta(u,v)+\rho(x,v)+\mu(u,y)+\theta(x,y)).
\end{equation}
\begin{rmk}
We don't have a map from $A'\otimes A'$ to $A$ because $A'$ is a subalgebra of $(\huaA,\cdot_{\huaA})$.
\end{rmk}
In fact, $[\Omega,\Omega]_G=0$ if and only if the following equations hold
\begin{equation}\label{associative-structure}
\left\{\begin{array}{ll}
& \frac{1}{2}[\pi,\pi]_G-(\xi\circ \theta)_2+(\eta\circ \theta)_1=0,\\
& (\rho\circ \pi)_1+\frac{1}{2}[\rho,\rho]_G-(\theta\circ\xi)_2+(\beta\circ\theta)_1=0,\\
&-(\mu\circ \pi)_2+\frac{1}{2}[\mu,\mu]_G+(\theta\circ\eta)_1-(\beta\circ\theta)_2=0,\\
&(\pi\circ \xi)_1-(\pi\circ \eta)_2+(\eta\circ \rho)_1-(\xi\circ \mu)_2=0,\\
&-(\pi\circ \xi)_2+(\xi\circ\pi)_1-(\xi\circ\rho)_2=0,\\
&(\pi\circ \eta)_1-(\eta\circ\pi)_2+(\eta\circ\mu)_1=0,\\
&\theta\circ\pi-(\rho\circ\theta)_2+(\mu\circ\theta)_1=0,\\
&[\rho,\mu]_G+(\theta\circ\xi)_1-(\theta\circ\eta)_2=0,\\
&(\rho\circ\xi)_1+(\beta\circ\rho)_1-(\rho\circ\beta)_2=0,\\
&(\rho\circ\eta)_1-(\beta\circ\rho)_2-(\mu\circ\xi)_2+(\beta\circ\mu)_1=0,\\
&-(\mu\circ\eta)_2+(\mu\circ\beta)_1-(\beta\circ\mu)_2=0,\\
&\frac{1}{2}[\xi,\xi]_G-(\xi\circ\beta)_2=0,\\
&[\xi,\eta]_G=0,\\
& \frac{1}{2}[\eta,\eta]_G+(\eta\circ\beta)_1=0,\\
&[\beta,\beta]_G=0,
\end{array}\right.
\end{equation}
%where
 %\begin{equation*}\left\{\begin{array}{ll}
%& (\theta\circ\xi)_1:\g\otimes\h\otimes\g\longrightarrow\h, (\theta\circ\xi)_2:\g\otimes\g\otimes\h\longrightarrow\h,\\
%& (\beta\circ\theta)_1:\g\otimes\g\otimes\h\longrightarrow\h, (\beta\circ\theta)_2:\h\otimes\g\otimes\g\longrightarrow\h,\\
%&(\theta\circ\eta)_1:\h\otimes\g\otimes\g\longrightarrow\h, (\theta\circ\eta)_2:\g\otimes\h\otimes\g\longrightarrow\h,\\
%&(\pi\circ\xi)_1:\g\otimes\h\otimes\g\longrightarrow\g, (\pi\circ\xi)_2:\g\otimes\g\otimes\h\longrightarrow\g,\\
%&(\pi\circ\eta)_1:\h\otimes\g\otimes\g\longrightarrow\g, (\pi\circ\eta)_2:\g\otimes\h\otimes\g\longrightarrow\g,\\
%&(\beta\circ\rho)_1:\g\otimes\h\otimes\h\longrightarrow\h, (\beta\circ\rho)_2:\h\otimes\g\otimes\h\longrightarrow\h,\\
%&(\beta\circ\mu)_1:\h\otimes\g\otimes\h\longrightarrow\h, (\beta\circ\mu)_2:\h\otimes\h\otimes\g\longrightarrow\h.,
%\end{array}\right.
%\end{equation*}
%are linear maps \yh{not clear}.

In the sequel, we consider various examples of quasi-twilled associative algebras.
\begin{ex}\label{direct-sum}
Let $(A,\cdot_{A})$ be an associative algebra. For all $\lambda \in \mathbb{K}$, define a bilinear operation $``\cdot_M$'' on $A\oplus A$ by
\begin{equation}
(x,u)\cdot_M(y,v)=(x\cdot_{A}v+u\cdot_{A} y,\lambda(x\cdot_{A}y)+u\cdot_{A}v),\quad
 \forall x,y,u,v \in A.
\end{equation}
Then $(A\oplus A,\cdot_M)$ is an associative algebra. Denote this associative algebra by $A\oplus_M A$. Moreover, $(A\oplus_M A,A,A)$ is a quasi-twilled associative algebra.
\end{ex}

\begin{ex}\label{semi-direct product}
Let $(V;\rho,\mu)$ be a representation of an associative algebra $(A,\cdot_{A})$. Define a bilinear operation $``\cdot_{(\rho,\mu)}$'' on $A\oplus V$ by
\begin{equation}
(x,u)\cdot_{(\rho,\mu)}(y,v)=(x\cdot_{A}y,\rho(x)v+\mu(y)u),\quad
 \forall x,y\in A, u,v \in V.
\end{equation}
Then $(A\oplus V,\cdot_{(\rho,\mu)})$ is an associative algebra, which is denoted by $A\ltimes_{(\rho,\mu)}V$ and called the {\bf semi-direct product} of the associative algebra $(A,\cdot_{A})$ and the representation $(V;\rho,\mu)$. Moreover, $(A\ltimes_{(\rho,\mu)}V,A,V)$ is a quasi-twilled associative algebra.
\end{ex}

\begin{ex}\label{semi-direct product-1}
Let $(A';\rho,\mu)$ be an associative representation of an associative algebra $(A,\cdot_{A})$. Define a bilinear operation $``\cdot_{(\rho,\mu)}$'' on $A\oplus A'$ by
\begin{equation}
(x,u)\cdot_{(\rho,\mu)}(y,v)=(x\cdot_{A}y,\rho(x)v+\mu(y)u+u\cdot_{A'}v),\quad
 \forall x,y\in A, u,v \in A'.
\end{equation}
Then $(A\oplus A',\cdot_{(\rho,\mu)})$ is an associative algebra, which is denoted by $A\ltimes_{(\rho,\mu)}A'$ and called the {\bf semi-direct product} of the associative algebra $(A,\cdot_A)$ and the associative representation $(A';\rho,\mu)$. Moreover, $(A\ltimes_{(\rho,\mu)}A',A,A')$ is a quasi-twilled associative algebra.
\end{ex}

\begin{ex}\label{direct product}
Let $(A,\cdot_{A})$ and $(A',\cdot_{A'})$ be associative algebras. Define a bilinear operation $``\cdot_{\oplus}$'' on $A\oplus A'$ by
\begin{equation}
(x,u)\cdot_{\oplus}(y,v)=(x\cdot_{A}y,u\cdot_{A'}v),\quad
 \forall x,y\in A, u,v \in A'.
\end{equation}
Then $(A\oplus A',\cdot_{\oplus})$ is an associative algebra. Moreover, $(A\oplus A',A,A')$ is a quasi-twilled associative algebra.

\end{ex}

\begin{ex}\label{abelian-extension}
Let $(V;\rho,\mu)$ be a representation of an associative algebra $(A,\cdot)$ and $\omega\in \Hom(\otimes^2A,V)$ a $2$-cocycle. Define a bilinear operation $``\cdot_{(\rho,\mu,\omega)}$'' on $A\oplus V$ by
\begin{equation}
(x,u)\cdot_{(\rho,\mu,\omega)}(y,v)=(x\cdot_{A}y,\rho(x)v+\mu(y)u+\omega(x,y)),\quad
 \forall x,y\in A, u,v \in V.
\end{equation}
Then $(A\oplus V,\cdot_{(\rho,\mu,\omega)})$ is an associative algebra. Denote this associative algebra by $A\ltimes_{(\rho,\mu,\omega)}V$. Moreover, $(A\ltimes_{(\rho,\mu,\omega)}V,A,V)$ is a quasi-twilled associative algebra.
\end{ex}
\begin{ex}\label{ex:Reynolds}
As a special case of Example \ref{abelian-extension}, consider the regular representation $(A;L,R)$ and $\omega(x,y)=x\cdot_{A}y$. Then we obtain a quasi-twilled associative algebra $(A\ltimes_{(L,R,\omega)} A,A,A)$.
\end{ex}
\begin{rmk}\label{rmk-nex}
 In fact, the above examples can be unified via extensions of associative algebras. Recall that an associative algebra  $(\huaA,\cdot_{\huaA})$ is an extension of an associative algebra $(A,\cdot_{A})$ by an associative algebra $(A',\cdot_{A'})$ if we have the following exact sequence:
\begin{equation}\label{seq:nonabelianext}
0\longrightarrow A'\stackrel{i}{\longrightarrow}\huaA\stackrel{\pi}{\longrightarrow} A\longrightarrow 0.
\end{equation}
Choosing a section $s:A\lon\huaA$, then $\huaA$ is equal to $s(A)\oplus i(A')$ and $i(A')$ is a subalgebra. Thus, $(\huaA, s(A), i(A'))$ is a quasi-twilled associative algebra.
\end{rmk}
\begin{ex}\label{matched-pair}
A {\bf matched pair of associative algebras} $(A,A',\rho,\mu,\eta,\xi)$, simply $(A,A')$, consists of two associative algebras $(A,\cdot_{A})$ and $(A',\cdot_{A'})$, together with linear maps $\rho,\mu:A\longrightarrow \gl(A')$ and $\eta,\xi:A'\longrightarrow \gl(A)$ such that $(A';\rho,\mu)$ and $(A;\eta,\xi)$ are representations, and for all $x,y\in A, u,v\in A'$, satisfying the following conditions:
\begin{eqnarray*}
\rho(x)(u\cdot_{A'}v)&=&\rho(\xi(u)x)v+(\rho(x)u)\cdot_{A'}v,\\
\mu(x)(u\cdot_{A'}v)&=&\mu(\eta(v)x)u+u\cdot_{A'}(\mu(x)v),\\
\eta(u)(x\cdot_{A}y)&=&\eta(\mu(x)u)y+(\eta(u)x)\cdot_{A}y,\\
\xi(u)(x\cdot_{A}y)&=&\xi(\rho(y)u)x+x\cdot_{A}(\xi(u)y),\\
\rho(\eta(u)x)v+(\mu(x)u)\cdot_{A'} v&=&\mu(\xi(v)x)u+u\cdot_{A'}(\rho(x)v),\\
\eta(\rho(x)u)y+(\xi(u)x)\cdot_{A} y&=&\xi(\mu(y)u)x+x\cdot_{A}(\eta(u)y).
\end{eqnarray*}
We define a bilinear operation $\cdot_{\bowtie}:\otimes^2(A\oplus A')\longrightarrow A\oplus A'$ by
\begin{equation}
(x,u)\cdot_{\bowtie}(y,v)=(x\cdot_{A}y+\xi(v)x+\eta(u)y,u\cdot_{A'}v+\rho(x)v+\mu(y)u),\quad
 \forall x,y\in A, u,v \in A'.
\end{equation}
Let $(A,A',\rho,\mu,\eta,\xi)$ be a matched pair of associative algebras, then $(A\oplus A',\cdot_{\bowtie})$ is an associative algebra. Denote this associative algebra by $A\bowtie A'$. Moreover, $(A\bowtie A',A,A')$ is a quasi-twilled associative algebra.
\end{ex}
In next two sections, $(\huaA,A,A')$ is always a quasi-twilled associative algebra and the multiplication on $\huaA$ is denoted by
\begin{equation*}
\Omega=\pi+\xi+\eta+\beta+\rho+\mu+\theta,
\end{equation*}
where $\pi \in \Hom(A\otimes A,A)$, $\xi \in \Hom(A\otimes A',A)$, $\eta \in \Hom(A'\otimes A,A)$, $\beta \in \Hom(A'\otimes A',A')$, $\rho \in \Hom(A\otimes A',A')$, $\mu \in \Hom(A'\otimes A,A')$, $\theta \in \Hom(A\otimes A,A')$.

\section{The controlling algebra and cohomology of right deformation maps of quasi-twilled associative algebras}\label{sec:deformation map-1}

\subsection{Right deformation maps of quasi-twilled associative algebras}
\
\newline
\indent
In this subsection, we introduce the notion of a right deformation map of a quasi-twilled associative algebra and give some examples, such as a modified Rota-Baxter operator of weight $\lambda$ on an associative algebra, a derivation on an associative algebra, a crossed homomorphism between associative algebras and an associative algebra homomorphism.
\begin{defi}
Let $(\huaA,A,A')$ be a quasi-twilled associative algebra. A {\bf right deformation map} of $(\huaA,A,A')$ is a linear map $D:A\longrightarrow A'$ such that for all $x,y\in A$
\begin{equation}\label{deformation-map}
D(\pi(x,y)+\xi(x,D(y))+\eta(D(x),y))=\rho(x,D(y))+\mu(D(x),y)+\beta(D(x),D(y))+\theta(x,y).
\end{equation}
\end{defi}
Let $(\huaA,A,A')$ be a quasi-twilled associative algebra and $D:A\longrightarrow A'$ a linear map. Denote the graph of $D$ by
\begin{equation*}
 \mathrm{Gr}(D)=\{(x,D(x))|x\in A\}.
\end{equation*}
\begin{pro}
With the above notation, $D$ is a right deformation map of $(\huaA,A,A')$ if and only if $\mathrm{Gr}(D)$ is a subalgebra of $\huaA$. Furthermore, $(\mathrm{Gr}(D),A')$ is a matched pair of associative algebras.
\end{pro}
\begin{proof}
If $D$ is a right deformation map of $(\huaA,A,A')$, for all $(x,D(x)),(y,D(y)) \in \mathrm{Gr}(D)$, by \eqref{deformation-map}, we have
\begin{eqnarray*}&&\Omega((x,D(x)),(y,D(y)))\\
&=&(\pi(x,y)+\xi(x,D(y))+\eta(D(x),y),\rho(x,D(y))+\mu(D(x),y)+\beta(D(x),D(y))+\theta(x,y))\\
&=&(\pi(x,y)+\xi(x,D(y))+\eta(D(x),y),D(\pi(x,y)+\xi(x,D(y))+\eta(D(x),y))),
\end{eqnarray*}
which implies that $\Omega((x,D(x)),(y,D(y)))\in \mathrm{Gr}(D)$. Thus, $\mathrm{Gr}(D)$ is a subalgebra of $\huaA$. The converse part can be proved similarly. We omit details.

Since $\huaA=\mathrm{Gr}(D)\oplus A'$, thus, $(\mathrm{Gr}(D),A')$ is a matched pair of associative algebras.
\end{proof}

\begin{ex}
(modified Rota-Baxter operator) Consider the quasi-twilled associative algebra $(A\oplus_MA,A,A)$ given in Example \ref{direct-sum}. In this case, %deformation map of type \uppercase\expandafter{\romannumeral1}%
a right deformation map of $(A\oplus_MA,A,A)$ is a linear map $D:A\longrightarrow A$ such that
\begin{equation}
D(x)\cdot_{A}D(y)-D(x\cdot_{A}D(y))-D(D(x)\cdot_{A} y)=-\lambda(x\cdot_{A} y),\quad
 \forall x,y\in A,
\end{equation}
which implies that $D$ is a {\bf modified Rota-Baxter operator of weight $\lambda$} on the associative algebra $(A,\cdot_{A})$. See \cite{Das1,Guo1,Guo2} for more details.
\end{ex}

\begin{ex}
(derivation) Consider the quasi-twilled associative algebra $(A\ltimes_{(\rho,\mu)}V,A,V)$ given in Example \ref{semi-direct product}. In this case, a right deformation map of $(A\ltimes_{(\rho,\mu)}V,A,V)$ is a linear map $D:A\longrightarrow V$ such that
\begin{equation}
D(x\cdot_{A}y)=\rho(x)D(y)+\mu(y)D(x),\quad
 \forall x,y\in A,
\end{equation}
which implies that $D$ is a {\bf derivation} from the associative algebra $(A,\cdot_{A})$ to $V$. In particular, if $(V;\rho,\mu)$ is a regular representation, then we obtain the usual derivation. See \cite{Loday} for more details.
\end{ex}

\begin{ex}
(crossed homomorphism) Consider the quasi-twilled associative algebra $(A\ltimes_{(\rho,\mu)}A',A,A')$ given in Example \ref{semi-direct product-1}. In this case, a right deformation map of $(A\ltimes_{(\rho,\mu)}A',A,A')$ is a linear map $D:A\longrightarrow A'$ such that
\begin{equation}
D(x\cdot_{A}y)=\rho(x)D(y)+\mu(y)D(x)+D(x)\cdot_{A'}D(y),\quad
 \forall x,y\in A,
\end{equation}
which implies that $D$ is a {\bf crossed homomorphism} from the associative algebra $(A,\cdot_{A})$ to the associative algebra $(A',\cdot_{A'})$. See \cite{Das3} for more details.
\end{ex}

\begin{ex}
(algebra homomorphism) Consider the quasi-twilled associative algebra $(A\oplus A',A,A')$ given in Example \ref{direct product}. In this case, a right deformation map of $(A\oplus A',A,A')$ is a linear map $D:A\longrightarrow A'$ such that
\begin{equation}
D(x\cdot_{A}y)=D(x)\cdot_{A'} D(y),\quad
 \forall x,y\in A,
\end{equation}
which implies that $D$ is an {\bf associative algebra homomorphism} from $(A,\cdot_{A})$ to $(A',\cdot_{A'})$.
\end{ex}

%\begin{ex}
%Consider the quasi-twilled associative algebra $(\g\ltimes_{(\rho,\mu,\omega)} V,\g,V)$ given in Example \ref{abelian-extension} obtained from the representations $\rho$ and $\mu$ of $\g$ on $V$ and a $2$-cocycle $\omega$. In this case, a deformation map of $(\g\ltimes_{(\rho,\mu,\omega)} V,\g,V)$ is a linear map $D:\g\longrightarrow V$ such that
%\begin{equation}\label{twisted-derivation}
%D(x\cdot_{\g}y)=\rho(x)D(y)+\mu(y)D(x)+\omega(x,y),\quad
 %\forall x,y\in \g.
%\end{equation}
%\end{ex}
%Thus, we give the following definition:
%\begin{defi}
%Let $\rho,\mu:\g\longrightarrow\gl(V)$ be the representations of an associative algebra $(\g,\cdot_{\g})$ on a vector space $V$ and $\omega\in \Hom(\otimes^2\g,V)$ is a $2$-cocycle. A linear map $D:\g\longrightarrow V$ satisfying \eqref{twisted-derivation} is called {\bf twisted derivation} from $(\g,\cdot_{\g})$ to $V$ with respect to the $2$-cocycle $\omega$.
%\end{defi}

Let $D:A\longrightarrow A'$ be a linear map. We set
\begin{equation*}
 e^{X_D}(\cdot):={\Id}+X_D+\frac{1}{2!}X_D^2+\frac{1}{3!}X_D^3+\dots,
\end{equation*}
where $X_D=[\cdot,D]_G$, $X_D^2:=[[\cdot,D]_G,D]_G$ and $X_D^n$ is defined by the same manner.
\begin{defi}{\rm(\cite{Uchino})}
 The transformation $\Omega^D\triangleq e^{X_D}(\Omega)$ is called a {\bf twisting} of $\Omega$ by $D$.
\end{defi}

\begin{pro}\label{pro:twisting-ass}{\rm(\cite{Uchino})}
Let $(\huaA,A,A')$ be a quasi-twilled associative algebra and $D:A\longrightarrow A'$ a linear map. Then
\begin{equation*}
\Omega^D=e^{-D}\circ \Omega\circ (e^D\otimes e^D)
\end{equation*}
is an associative algebra structure on $\huaA$ and $e^D$ is an associative algebra isomorphism from $(\huaA,\Omega^D)$ to $(\huaA,\Omega)$. Furthermore, $((\huaA,\Omega^D),A,A')$ is a quasi-twilled associative algebra.
\end{pro}
Let $(\huaA,A,A')$ be a quasi-twilled associative algebra and $D:A\longrightarrow A'$  a linear map. Write
\begin{equation*}
\Omega^D=\pi^D+\xi^D+\eta^D+\beta^D+\rho^D+\mu^D+\theta^D,
\end{equation*}
where $\pi^D \in \Hom(A\otimes A,A)$, $\xi^D \in \Hom(A\otimes A',A)$, $\eta^D \in \Hom(A'\otimes A,A)$, $\beta^D \in \Hom(A'\otimes A',A')$, $\rho^D \in \Hom(A\otimes A',A')$, $\mu^D \in \Hom(A'\otimes A,A')$, $\theta^D \in \Hom(A\otimes A,A')$. Then we have the following result:
\begin{thm}\label{twisting-ass-deformation-1}
With the above notation, for all $x,y\in A, u,v\in A'$, we have
\begin{eqnarray}
\label{def-1}\pi^D(x,y)&=&\pi(x,y)+\eta(D(x),y)+\xi(x,D(y)),\\
\label{def-4}\xi^D(x,u)&=&\xi(x,u),\\
\label{def-5}\eta^D(u,x)&=&\eta(u,x),\\
\label{def-6}\beta^D(u,v)&=&\beta(u,v),\\
\label{def-2}\rho^D(x,u)&=&\rho(x,u)+\beta(D(x),u)-D(\xi(x,u)),\\
\label{def-3}\mu^D(u,x)&=&\mu(u,x)+\beta(u,D(x))-D(\eta(u,x)),\\
\label{def-7}\theta^D(x,y)&=&\theta(x,y)+\mu(D(x),y)+\rho(x,D(y))-D(\pi(x,y))\\
\nonumber &&+\beta(D(x),D(y))-D(\eta(D(x),y))-D(\xi(x,D(y))).
\end{eqnarray}
Furthermore, $D:A\longrightarrow A'$ is a right deformation map of $(\huaA,A,A')$ if and only if $(A,\pi^D)$ and $(A',\beta)$ are associative algebra, thus, they constitute a matched pair.
\end{thm}
\begin{proof}
For all $x,y\in A, u,v\in A'$, we have
\begin{eqnarray*}
(\pi^D(x,y),\theta^D(x,y))&=&\Omega^D((x,0),(y,0))\\
&=&\Omega((x,0),(y,0))+\Omega((0,D(x)),(y,0))+\Omega((x,0),(0,D(y)))\\
&&-D(\Omega((x,0),(y,0)))+\Omega((0,D(x)),(0,D(y)))-D(\Omega((0,D(x)),(y,0)))\\
&&-D(\Omega((x,0),(0,D(y))))-D(\Omega((0,D(x)),(0,D(y))))\\
&=&(\pi(x,y),\theta(x,y))+(\eta(D(x),y),\mu(D(x),y))+(\xi(x,D(y)),\rho(x,D(y)))\\
&&-(0,D(\pi(x,y)))+(0,\beta(D(x),D(y)))-(0,D(\eta(D(x),y)))-(0,D(\xi(x,D(y)))\\
&=&(\pi(x,y)+\eta(D(x),y)+\xi(x,D(y)),\theta(x,y)+\mu(D(x),y)+\rho(x,D(y))\\
&&-D(\pi(x,y))+\beta(D(x),D(y))-D(\eta(D(x),y))-D(\xi(x,D(y))),
\end{eqnarray*}
which implies that
\begin{eqnarray*}
\pi^D(x,y)&=&\pi(x,y)+\eta(D(x),y)+\xi(x,D(y)),\\
\theta^D(x,y)&=&\theta(x,y)+\mu(D(x),y)+\rho(x,D(y))-D(\pi(x,y))\\
&&+\beta(D(x),D(y))-D(\eta(D(x),y))-D(\xi(x,D(y))).
\end{eqnarray*}
Similarly, we have
\begin{eqnarray*}
\xi^D(x,u)&=&\xi(x,u),\\
\eta^D(u,x)&=&\eta(u,x),\\
\beta^D(u,v)&=&\beta(u,v)\\
\rho^D(x,u)&=&\rho(x,u)+\beta(D(x),u)-D(\xi(x,u)),\\
\mu^D(u,x)&=&\mu(u,x)+\beta(u,D(x))-D(\eta(u,x)).
\end{eqnarray*}
If $D$ is a right deformation map of $(\huaA,A,A')$, then we have $\theta^D=0$.  Since $\Omega^D$ is an associative algebra structure on $\huaG$, we have $[\Omega^D,\Omega^D]_G=0$.  Thus, by \eqref{associative-structure}, we have $[\pi^D,\pi^D]_G=0$ and $[\beta^D,\beta^D]_G=0$, which implies that $(A,\pi^D)$ and $(A',\beta)$ are associative algebras. Thus, $(A,\pi^D)$ and $(A',\beta)$ constitutes a matched pair. The converse part can be proved similarly. We omit details. This finishes the proof.
\end{proof}
\subsection{Controlling algebras of right deformation maps of quasi-twilled associative algebras}
\
\newline
\indent
In this subsection, first we recall the notion of a curved $L_\infty$-algebra and a curved $V$-data. Then we construct a curved $L_\infty$-algebra from a quasi-twilled associative algebra. Finally, we give the controlling algebra of right deformation maps of quasi-twilled associative algebras, which is a curved $L_\infty$-algebra. An important byproduct is the controlling algebra of modified Rota-Baxter operators of weight $\lambda$ on an associative algebra, which is totally unknown before.

Let $\g=\oplus_{k\in\mathbb Z}\g^k$ be a $\mathbb Z$-graded vector space. The suspension operator $s$ changes the grading of $\g$ according to the rule $(s\g)^i:=\g^{i-1}$. The degree $1$ map $s:\g\longrightarrow s\g$ is defined by sending $x\in \g$ to its copy $s x\in s\g$.

\begin{defi}\rm(\cite{KS})
Let $\g=\oplus_{k\in\mathbb Z}\g^k$ be a $\mathbb Z$-graded vector space. A {\bf  curved $L_\infty$-algebra} is a $\mathbb Z$-graded vector space $\g$ equipped with a collection $(k\ge 0)$ of linear maps $l_k:\otimes^k\g\lon\g$ of degree $1$ with the property that, for any homogeneous elements $x_1,\cdots, x_n\in \g$, we have
\begin{itemize}\item[\rm(i)]
{\em (graded symmetry)} for every $\sigma\in S_{n}$,
\begin{eqnarray*}
l_n(x_{\sigma(1)},\cdots,x_{\sigma(n)})=\varepsilon(\sigma)l_n(x_1,\cdots,x_n),
\end{eqnarray*}
\item[\rm(ii)] {\em (generalized Jacobi identity)} for all $n\ge 0$,
\begin{eqnarray*}\label{sh-Lie}
\sum_{i=0}^{n}\sum_{\sigma\in  S(i,n-i) }\varepsilon(\sigma)l_{n-i+1}(l_i(x_{\sigma(1)},\cdots,x_{\sigma(i)}),x_{\sigma(i+1)},\cdots,x_{\sigma(n)})=0,
\end{eqnarray*}

\end{itemize}where $\varepsilon(\sigma)=\varepsilon(\sigma;x_1,\cdots, x_n)$ is the   Koszul sign for a permutation $\sigma\in S_n$ and $x_1,\cdots, x_n\in \g$.
\end{defi}
We denote a curved $L_\infty$-algebra by $(\g,\{l_k\}_{k=0}^{+\infty})$. A curved $L_\infty$-algebra $(\g,\{l_k\}_{k=0}^{+\infty})$ with $l_0=0$ is exactly an $L_\infty$-algebra~\cite{LS}.

\begin{defi}
Let $(\g, \{l_k\}_{k=0}^{+\infty})$ be a curved $L_\infty$-algebra. A {\bf Maurer-Cartan element} is a degree $0$ element $x$ satisfying
\begin{equation*}
 l_0+\sum_{k=1}^{+\infty}\frac{1}{k!}l_k(x, \cdots, x)=0.
\end{equation*}
%where $l_0=\Phi\in\g^1$.
\end{defi}

Let $x$ be a Maurer-Cartan element of a curved $L_\infty$-algebra $(\g, \{l_k\}_{k=0}^{+\infty})$. Define $l_{k}^{x}:\otimes^{k} \g\lon \g ~~(k\geq1)$ by
\begin{equation*}
l_{k}^{x}(x_1,\cdots,x_k)=\sum_{n=0}^{+\infty}\frac{1}{n!}l_{k+n}(\underbrace{x,\cdots,x}_{n},x_1,\cdots, x_k).
\end{equation*}

\begin{thm}\label{twistLin}{\rm(\cite{DSV,Get})}
With the above notation,  $(\g,\{l_k^{x}\}_{k=1}^{+\infty})$ is an $L_\infty$-algebra, which is called the {\bf twisted $L_\infty$-algebra} by $x$.
\end{thm}
We recall Voronov's derived bracket construction \cite{Vo}, which is a powerful method for constructing a curved $L_\infty$-algebra.

\begin{defi}\rm(\cite{Vo})
A {\bf curved $V$-data} consists of a quadruple $(L,F,P,\Delta)$, where
\begin{itemize}
\item[$\bullet$] $(L=\oplus L^i,[\cdot,\cdot])$ is a graded Lie algebra,
\item[$\bullet$] $F$ is an abelian graded Lie subalgebra of $(L,[\cdot,\cdot])$,
\item[$\bullet$] $P:L\lon L$ is a projection, that is $P\circ P=P$, whose image is $F$ and kernel is a  graded Lie subalgebra of $(L,[\cdot,\cdot])$,
\item[$\bullet$] $\Delta$ is an element in $L^1$ such that $[\Delta,\Delta]=0$.
\end{itemize}
When $\Delta\in\ker(P)^{1}$ such that $[\Delta, \Delta]=0$, we refer to $(L,F,P,\Delta)$ as a {\bf $V$-data}.
\end{defi}

\begin{thm}\rm(\cite{Vo})\label{cV}
Let $(L,F,P,\Delta)$ be a curved $V$-data. Then  $(F,\{l_k\}_{k=0}^{+\infty})$ is a curved $L_\infty$-algebra, where $l_k$ are given by
\begin{equation*}
l_0=P(\Delta), \quad l_k(x_1, \cdots, x_n)=P([\cdots[[\Delta, x_1], x_2], \cdots, x_n]).
\end{equation*}
\end{thm}
\begin{thm}\label{curved-L}
Let $(\huaA,A,A')$ be a quasi-twilled associative algebra. Then there is a curved $V$-data $(L,F,P,\Delta)$ as follows:
\begin{itemize}
\item[$\bullet$] the graded Lie algebra $(L,[\cdot,\cdot])$ is given by $(\oplus_{n=0}^{+\infty}\Hom(\otimes^{n+1}(A\oplus A'), A\oplus A'), [\cdot, \cdot]_G)$,
\item[$\bullet$] the abelian graded Lie subalgebra $F$ is given by $\oplus_{n=0}^{+\infty}\Hom(\otimes^{n+1}A, A')$,
\item[$\bullet$] $P:L\lon L$ is the projection onto the subspace $F$,
\item[$\bullet$] $\Delta=\pi+\xi+\eta+\beta+\rho+\mu+\theta$.
\end{itemize}
Consequently, we obtain a curved $L_{\infty}$-algebra $(\oplus_{n=0}^{+\infty}\Hom(\otimes^{n+1}A, A'), l_0, l_1, l_2)$, where $l_0,l_1,l_2$ are given by
\begin{eqnarray*}
l_0&=&\theta\\
l_1(f)&=&[\pi+\rho+\mu, f]_G\\
l_2(f, g)&=&[[\xi+\eta+\beta, f]_G, g]_G,
\end{eqnarray*}
for all $f\in \Hom(\otimes^mA,A')$ and $g\in \Hom(\otimes^nA,A')$.

Furthermore, a linear map $D:A\longrightarrow A'$ is a right deformation map of $(\huaA, A, A')$ if and only if $D$ is a Maurer-Cartan element of the above curved $L_{\infty}$-algebra.
\end{thm}
\begin{proof}
It is obvious  that $F$ is an abelian graded Lie subalgebra of $L$. Since $P$ is the projection onto $F$, we have $P^2=P$. Obviously, the kernel of $P$ is a graded Lie subalgebra of $L$ and $[\Delta,\Delta]_G=0$. Thus, $(L,F,P,\Delta)$ is a curved $V$-data. By Theorem \ref{cV}, for all $f\in \Hom(\otimes^mA,A')$, $g\in \Hom(\otimes^nA,A')$, by direct computation, we have
\begin{eqnarray*}
l_0&=&P(\Delta)=\theta,\\
l_1(f)&=&P([\pi+\xi+\eta+\beta+\rho+\mu+\theta, f]_G)=[\pi+\rho+\mu, f]_G,\\
l_2(f, g)&=&P([[\pi+\xi+\eta+\beta+\rho+\mu+\theta, f]_G, g]_G)=[[\xi+\eta+\beta, f]_G, g]_G.
\end{eqnarray*}
By direct computation, we also have
$$
[[\pi+\xi+\eta+\beta+\rho+\mu+\theta, f]_G, g]_G\in\Hom(\otimes^{m+n}A, A').
$$
Since $F$ is abelian, we have $l_k=0$ for all $k\geq 3$.
Moreover, for all $x,y\in A$, we have
\begin{eqnarray*}
  &&l_0(x, y)+l_1(D)(x, y)+\frac{1}{2}l_2(D, D)(x, y)\\
  &=&\theta(x, y)+[\pi+\rho+\mu, D]_G(x, y)+\frac{1}{2}[[\xi+\eta+\beta, D]_G, D]_G(x, y)\\
  &=&\theta(x, y)+\rho(x, D(y))+\mu(D(x),y)-D(\pi(x, y))+\beta(D(x), D(y))-D(\xi(x, D(y)))-D(\eta(D(x),y)).
\end{eqnarray*}
Thus, $D$ is a right deformation map of $(\huaA, A, A')$ if and only if $D$ is a Maurer-Cartan element of the curved $L_\infty$-algebra $(\oplus_{n=0}^{+\infty}\Hom(\otimes^{n+1}A, A'), l_0, l_1, l_2)$.
\end{proof}

\begin{cor}
Consider the quasi-twilled associative algebra $(A\oplus_MA,A,A)$ given in Example \ref{direct-sum}. Then $(\oplus_{n=0}^{+\infty}\Hom(\otimes^{n+1}A, A), l_0, l_1, l_2)$ is a curved $L_\infty$-algebra, where $l_0, l_1, l_2$ are given by $l_0(x,y)=\lambda (x\cdot_{A}y)$, $l_1=0$,
 and
\begin{eqnarray*}
&&l_2(f, g)(x_1,\dots,x_{m+n})\\
&=&(-1)^{m-1}g(x_1,\dots,x_n)\cdot_{A}f(x_{n+1},\dots,x_{m+n})+(-1)^{m(n-1)}f(x_1,\dots,x_m)\cdot_{A}g(x_{m+1},\dots,x_{m+n})\\
&&-\sum_{i=1}^m(-1)^{in}(-1)^mf(x_1,\dots,x_{i-1},x_i\cdot_{A}g(x_{i+1},\dots,x_{i+n}),x_{i+n+1},\dots,x_{m+n})\\
&&+\sum_{i=1}^m(-1)^{(i-1)n}(-1)^mf(x_1,\dots,x_{i-1},g(x_i,\dots,x_{i+n-1})\cdot_{A}x_{i+n},x_{i+n+1},\dots,x_{m+n})\\
&&-\sum_{i=1}^n(-1)^{(i-1)m}(-1)^{m(n-1)}g(x_1,\dots,x_{i-1},f(x_i,\dots,x_{i+m-1})\cdot_{A}x_{i+m},x_{i+m+1},\dots,x_{m+n})\\
&&+\sum_{i=1}^n(-1)^{(i-1)m}(-1)^{mn}g(x_1,\dots,x_{i-1},x_i\cdot_{A}f(x_{i+1},\dots,x_{i+m}),x_{i+m+1},\dots,x_{m+n}),
\end{eqnarray*}
for all $f\in\Hom(\otimes^mA, A), g\in\Hom(\otimes^nA, A)$ and $x,y,x_1,\dots,x_{m+n}\in A$.

Moreover, Maurer-Cartan elements of the curved $L_\infty$-algebra $(\oplus_{n=0}^{+\infty}\Hom(\otimes^{n+1}A, A), l_0, l_1, l_2)$ are exactly modified Rota-Baxter operators of weight $\lambda$ on the associative algebra $(A,\cdot_{A})$. Thus, this curved $L_\infty$-algebra can be viewed as the {\bf controlling algebra of modified Rota-Baxter operators of weight $\lambda$} on the associative algebra $(A,\cdot_{A})$.
\end{cor}

\begin{cor}
Consider the quasi-twilled associative algebra $(A\ltimes_{(\rho,\mu)}V,A,V)$ given in Example \ref{semi-direct product}. Then the curved $L_\infty$-algebra $(\oplus_{n=0}^{+\infty}\Hom(\otimes^{n+1}A, V), l_0, l_1, l_2)$ is exactly a cochain complex $(\oplus_{n=0}^{+\infty}\Hom(\otimes^{n+1}A, V),\dM)$, where $l_0=l_2=0$ and $l_1=\dM$ is given by
\begin{eqnarray}
\label{derivation}\dM f(x_1,\dots,x_{m+1})&=&(-1)^{m-1}\rho(x_1)f(x_2,\dots,x_{m+1})+\mu(x_{m+1})f(x_1,\dots,x_m)\\
\nonumber &&+(-1)^{m-1}\sum_{i=1}^m(-1)^if(x_1,\dots,x_{i-1},x_i\cdot_{A}x_{i+1},x_{i+2},\dots,x_{m+1}),
\end{eqnarray}
for all $f\in\Hom(\otimes^mA, V), x_1,\dots,x_{m+1}\in A$.

Moreover, a linear map $D:A\longrightarrow V$ is a derivation from $(A,\cdot_{A})$ to $V$ if and only if $\dM(D)=0$. Thus, this cochain complex can be viewed as the {\bf controlling algebra of derivations} from the associative algebra $(A,\cdot_{A})$ to $V$. See \cite{Guo} for more details.
\end{cor}

\begin{cor}
Consider the quasi-twilled associative algebra $(A\ltimes_{(\rho,\mu)}A',A,A')$ given in Example \ref{semi-direct product-1}. Then we obtain a differential graded Lie algebra $(s(\oplus_{n=0}^{+\infty}\Hom(\otimes^{n+1}A, A')),\Courant{\cdot,\cdot},\dM)$, where the graded Lie bracket $\Courant{\cdot,\cdot}$ is given by
\begin{eqnarray}
\nonumber&&\Courant{f,g}(x_1,\dots,x_{m+n})\\
\nonumber&=&(-1)^{m-1}l_2(f,g)\\
\label{semi}&=&(-1)^{mn+1}f(x_1,\dots,x_m)\cdot_{A'}g(x_{m+1},\dots,x_{m+n})+g(x_1,\dots,x_n)\cdot_{A'}f(x_{n+1},\dots,x_{m+n}),
\end{eqnarray}
and the differential $\dM$ is given by \eqref{derivation} for all $f\in\Hom(\otimes^mA, A'), g\in\Hom(\otimes^nA, A')$, $x_1,\dots,x_{m+n}\in A$.

Moreover, Maurer-Cartan elements of $(s(\oplus_{n=0}^{+\infty}\Hom(\otimes^{n+1}A, A')),\Courant{\cdot,\cdot},\dM)$ are exactly crossed homomorphisms from the associative algebra $(A,\cdot_{A})$ to the associative algebra $(A',\cdot_{A'})$. Thus, this differential graded Lie algebra can be viewed as the {\bf controlling algebra of crossed homomorphisms} from the associative algebra$(A,\cdot_{A})$ to the associative algebra $(A',\cdot_{A'})$. See \cite{Das3} for more details.
\end{cor}

\begin{cor}
Consider the quasi-twilled associative algebra $(A\oplus A',A,A')$ given in Example \ref{direct product}. Then we obtain a differential graded Lie algebra
$(s(\oplus_{n=0}^{+\infty}\Hom(\otimes^{n+1}A, A')),\Courant{\cdot,\cdot},\dM)$, where the graded Lie bracket $\Courant{\cdot,\cdot}$ is given by \eqref{semi}
%\begin{eqnarray*}
%&&\Courant{f,g}(x_1,\dots,x_{m+n})\\
%&=&(-1)^{mn+1}f(x_1,\dots,x_m)\cdot_{A'}g(x_{m+1},\dots,x_{m+n})+g(x_1,\dots,x_n)\cdot_{A'}f(x_{n+1},\dots,x_{m+n}),
%\end{eqnarray*}
and the differential $\dM$ is given by
\begin{equation*}
\dM f(x_1,\dots,x_{m+1})=(-1)^{m-1}\sum_{i=1}^m(-1)^if(x_1,\dots,x_{i-1},x_i\cdot_{A}x_{i+1},x_{i+2},\dots,x_{m+1}),
\end{equation*}
for all $f\in\Hom(\otimes^mA,A'), g\in\Hom(\otimes^nA, A')$ and $x_1,\dots,x_{m+n}\in A$.

Moreover, Maurer-Cartan elements of $(s(\oplus_{n=0}^{+\infty}\Hom(\otimes^{n+1}A, A')),\Courant{\cdot,\cdot},\dM)$ are exactly associative algebra homomorphisms from $(A,\cdot_{A})$ to $(A',\cdot_{A'})$. Thus, this differential graded Lie algebra can be viewed as the {\bf controlling algebra of associative algebra homomorphisms} from $(A,\cdot_{A})$ to $(A',\cdot_{A'})$. See \cite{Barmeier,BD,Fregier,Gerstenhaber3} for more details.
\end{cor}

Let $(\huaA, A, A')$ be a quasi-twilled associative algebra and $D:A\longrightarrow A'$ a right deformation map of $(\huaA, A, A')$. By Theorem \ref{curved-L}, $D$ is a  Maurer-Cartan element of the curved $L_{\infty}$-algebra $(\oplus_{n=0}^{+\infty}\Hom(\otimes^{n+1}A, A'), l_0, l_1, l_2)$. By Theorem \ref{twistLin}, we have a twisted $L_\infty$-algebra structure on $\oplus_{n=0}^{+\infty}\Hom(\otimes^{n+1}A, A')$ as following:
\begin{eqnarray}
\label{twist-L-1}l_1^{D}(f)&=&l_1(f)+l_2(D,f),\\
\label{twist-L-2}l_2^{D}(f,g)&=&l_2(f,g),\\
\label{twist-L-3} l^D_k&=&0,\,\,\,\,k\ge3,
\end{eqnarray}
where $f\in\Hom(\otimes^mA, A'), g\in\Hom(\otimes^nA, A')$.

\begin{thm}\label{thm:MC-twist-L}
Let $D:A\longrightarrow A'$ be a right deformation map of a quasi-twilled associative algebra $(\huaA, A, A')$ and $D':A\longrightarrow A'$ a linear map. Then $D+D'$ is a right deformation map of $(\huaA, A, A')$ if and only if $D'$ is a Maurer-Cartan element of the twisted $L_\infty$-algebra $(\oplus_{n=0}^{+\infty}\Hom(\otimes^{n+1}A, A'), l^D_1, l^D_2)$.
\end{thm}
\begin{proof}
Let $D+D'$ be a right deformation map of $(\huaA, A, A')$, by Theorem \ref{curved-L}, we have
\begin{eqnarray*}
0&=&l_0+l_1(D+D')+\frac{1}{2}l_2(D+D',D+D')\\
&=&l_1(D')+l_2(D,D')+\frac{1}{2}l_2(D',D').
\end{eqnarray*}
which implies that
$$
l_1^{D}(D')+\frac{1}{2}l_2^{D}(D',D')=0,
$$
Thus, $D'$ is a Maurer-Cartan element of the twisted $L_\infty$-algebra $(\oplus_{n=0}^{+\infty}\Hom(\otimes^{n+1}A, A'), l^D_1, l^D_2)$.

The converse part can be proved similarly. We omit details. The proof is finished.
\end{proof}
\subsection{Cohomologies of right deformation maps of quasi-twilled associative algebras}
\
\newline
\indent
In this subsection, we introduce the cohomology of right deformation maps of quasi-twilled associative algebras. More precisely, we give the cohomology of a modified Rota-Baxter operator of weight $\lambda$, a derivation, a crossed homomorphism and an associative algebra homomorphism.

\begin{lem}
Let $D:A\longrightarrow A'$ be a right deformation maps of a quasi-twilled associative algebra $(\huaA,A,A')$. Then
$(A';\rho^D,\mu^D)$ is a representation of the associative algebra $(A,\pi^D)$, where $\pi^D,\rho^D,\mu^D$ are given by \eqref{def-1}, \eqref{def-2}, \eqref{def-3} respectively.
\end{lem}
\begin{proof}
Since $D$ is a right deformation map of $(\huaA,A,A')$, we have $\theta^D=0$, where $\theta^D$ is given by \eqref{def-7}. By Theorem \ref{twisting-ass-deformation-1}, we obtain that $(A,\pi^D)$ is an associative algebra.  By Proposition \ref{pro:twisting-ass}, $[\Omega^D,\Omega^D]=0$. Thus, by \eqref{associative-structure}, we have
\begin{eqnarray}
\label{twisting-cohomology-1}(\rho^D\circ\pi^D)_1+\frac{1}{2}[\rho^D,\rho^D]_G&=&0,\\
\label{twisting-cohomology-2}-(\mu^D\circ\pi^D)_2+\frac{1}{2}[\mu^D,\mu^D]_G&=&0,\\
\label{twisting-cohomology-3}[\rho^D,\mu^D]_G&=&0.
\end{eqnarray}
By \eqref{twisting-cohomology-1}, for all $x,y,z\in A, u,v,w\in A'$, we have
\begin{eqnarray*}
0&=&(\rho^D\circ\pi^D)_1((x,u),(y,v),(z,w))+\frac{1}{2}[\rho^D,\rho^D]_G((x,u),(y,v),(z,w))\\
&=&\rho^D(\pi^D((x,u),(y,v)),(z,w))+\rho^D(\rho^D((x,u),(y,v)),(z,w))-\rho^D((x,u),\rho^D((y,v),(z,w)))\\
&=&\rho^D(\pi^D(x,y),w)-\rho^D(x,\rho^D(y,w))\\
&=&\rho^D(\pi^D(x,y))(w)-\rho^D(x)\rho^D(y)(w),
\end{eqnarray*}
which implies that
\begin{equation}\label{rep-1}
\rho^D(\pi^D(x,y))=\rho^D(x)\circ\rho^D(y).
\end{equation}
Similarly, by \eqref{twisting-cohomology-2} and \eqref{twisting-cohomology-3}, we have
\begin{eqnarray}
\label{rep-2}\mu^D(\pi^D(x,y))&=&\mu^D(y)\circ\mu^D(x),\\
\label{rep-3}\rho^D(x)\circ\mu^D(y)&=&\mu^D(y)\circ\rho^D(x).
\end{eqnarray}
Thus, by \eqref{rep-1},\eqref{rep-2} and \eqref{rep-3}, we obtain that $(A';\rho^D,\mu^D)$ is a representation of $(A,\pi^D)$.
\end{proof}

Let $(A';\rho^D,\mu^D)$ be a representation of an associative algebra $(A,\pi^D)$. Define the set of $0$-cochains $C^{0}(D)$ to be $A'$. For $n\geq 1$, define the set of $n$-cochains $C^{n}(D)$ by $C^{n}(D)=\Hom(\otimes^nA,A')$. The coboundary operator $\dM^D:C^n(D)\longrightarrow C^{n+1}(D)$ is defined by
\begin{eqnarray*}
\dM^D f(x_1,\dots,x_{n+1})&=&\rho^D(x_1)f(x_2,\dots,x_{n+1})+\sum_{i=1}^n(-1)^if(x_1,\dots,x_{i-1},\pi^D(x_i, x_{i+1}),x_{i+2},\dots,x_{n+1})\\
&&+(-1)^{n+1}\mu^D(x_{n+1})f(x_1,\dots,x_n)\\
&=&\rho(x_1,f(x_2,\dots,x_{n+1}))+\beta(D(x_1),f(x_2,\dots,x_{n+1}))-D(\xi(x_1,f(x_2,\dots,x_{n+1})))\\
&&+\sum_{i=1}^n(-1)^if(x_1,\dots,x_{i-1},\pi(x_i, x_{i+1}),x_{i+2},\dots,x_{n+1})\\
&&+\sum_{i=1}^n(-1)^if(x_1,\dots,x_{i-1},\eta(D(x_i),x_{i+1}),x_{i+2},\dots,x_{n+1})\\
&&+\sum_{i=1}^n(-1)^if(x_1,\dots,x_{i-1},\xi(x_i, D(x_{i+1})),x_{i+2},\dots,x_{n+1})\\
&&+(-1)^{n+1}\mu(f(x_1,\dots,x_n),x_{n+1})+(-1)^{n+1}\beta(f(x_1,\dots,x_n),D(x_{n+1}))\\
&&-(-1)^{n+1}D(\eta(f(x_1,\dots,x_n),x_{n+1})),
\end{eqnarray*}
for all $f\in \Hom(\otimes^nA,A')$ and $x_1,\dots,x_{n+1}\in A$.
\begin{defi}
The cohomology of the cochain complex $(\oplus_{i=0}^{+\infty}C^i(D),\dM^D)$ is called the {\bf cohomology of right deformation maps} of the quasi-twilled associative algebra $(\huaA,A,A')$.
 \end{defi}
Denote the set of closed $n$-cochains by $Z^n(D)$ and the set of exact $n$-cochains by $B^n(D)$. We denote by $H^n(D)$=$Z^n(D)$/$B^n(D)$ the corresponding cohomology group.

\begin{pro}\label{cohomology-deformation-map}
Let $D:A\longrightarrow A'$ be a right deformation map of a quasi-twilled associative algebra $(\huaA,A,A')$ and $(\oplus_{n=0}^{+\infty}\Hom(\otimes^{n+1}A, A'), l^D_1, l^D_2)$ a twisted $L_\infty$-algebra. Then for all $f\in \Hom(\otimes^mA, A')$, we have
$$l_1^D(f)=(-1)^{m-1}\dM^{D}f.$$
\end{pro}
\begin{proof}
 By \eqref{twist-L-1} and Theorem \ref{curved-L}, we have
 \begin{eqnarray*}
l_1^D(f)&=&l_1(f)+l_2(D, f)\\
&=& [\pi+\rho+\mu, f]_G+[[\xi+\eta+\beta, D]_G, f]_G\\
&=&[\pi+\rho+\mu+[\xi+\eta+\beta, D]_G, f]_G.
\end{eqnarray*}
For all $x,y\in A, u,v\in A'$, by Theorem \ref{twisting-ass-deformation-1}, we have
 \begin{eqnarray*}
&&(\pi+\rho+\mu+[\xi+\eta+\beta, D]_G)((x,u),(y,v))\\
&=&(\pi(x,y),0)+(0,\rho(x,v))+(0,\mu(u,y))+(\xi+\eta+\beta)((0,D(x)),(y,v))\\
&&+(\xi+\eta+\beta)((x,u),(0,D(y)))-D((\xi+\eta+\beta)((x,u),(y,v)))\\
&=&(\pi(x,y),0)+(0,\rho(x,v))+(0,\mu(u,y))+(\eta(D(x),y),0)+(0,\beta(D(x),v))\\
&&+(\xi(x,D(y)),0)+(0,\beta(u,D(y)))-(0,D(\xi(x,v)))-(0,D(\eta(u,y)))\\
&=&(\pi^D(x,y),\rho^D(x,v)+\mu^D(u,y)).
\end{eqnarray*}
Thus, we have
\begin{equation*}
 l_1^D(f)=[\pi^D+\rho^D+\mu^D, f]_G= (-1)^{m-1}\dM^D(f).
\end{equation*}
This finishes the proof.
 \end{proof}

 \begin{ex}
Consider the quasi-twilled associative algebra $(A\oplus_MA,A,A)$ given in Example \ref{direct-sum}. Let $D:A\longrightarrow A$ be a modified Rota-Baxter operator of weight $\lambda$ on the associative algebra $(A,\cdot_{A})$. Then $(A,\pi^D)$ is an associative algebra, where $\pi^D$ is given by
\begin{equation*}
\pi^D(x,y)=D(x)\cdot_{A}y+x\cdot_{A} D(y),\quad
 \forall x,y\in A.
\end{equation*}
Moreover, $(A;\rho^D,\mu^D)$ is a representation of $(A,\pi^D)$, where $\rho^D$ and $\mu^D$ are given by
\begin{eqnarray*}
\rho^D(x)(y)&=&D(x)\cdot_{A}y-D(x\cdot_{A}y),\\
\mu^D(x)(y)&=&y\cdot_{A}D(x)-D(y\cdot_{A}x),
\end{eqnarray*}
for all $x,y \in A$.

The corresponding cohomology is taken to be the {\bf cohomology of a modified Rota-Baxter operator of weight $\lambda$} on the associative algebra $(A,\cdot_{A})$. See \cite{Das1} for more details.
\end{ex}
\begin{ex}
Consider the quasi-twilled associative algebra $(A\ltimes_{(\rho,\mu)}V,A,V)$ given in Example \ref{semi-direct product}. Let $D$ be a derivation from $(A,\cdot_{A})$ to $V$. Then $(A,\pi^D)$ is an associative algebra, where $\pi^D$ is given by
\begin{equation*}
\pi^D(x,y)=x\cdot_{A}y,\quad
 \forall x,y\in A.
\end{equation*}
Moreover, $(V;\rho^D,\mu^D)$ is a representation of $(A,\pi^D)$, where $\rho^D$ and $\mu^D$ are given by
\begin{eqnarray*}
\rho^D(x)(u)&=&\rho(x)(u),\\
\mu^D(x)(u)&=&\mu(x)(u),
\end{eqnarray*}
for all $x \in A, u\in V$.

The corresponding cohomology is taken to be the {\bf cohomology of a derivation} from $(A,\cdot_{A})$ to $V$. See \cite{Guo} for more details.
\end{ex}

\begin{ex}
Consider the quasi-twilled associative algebra $(A\ltimes_{(\rho,\mu)}A',A,A')$ given in Example \ref{semi-direct product-1}. Let $D$ be a crossed homomorphism from the associative algebra $(A,\cdot_{A})$ to the associative algebra $(A',\cdot_{A'})$. Then $(A,\pi^D)$ is an associative algebra, where $\pi^D$ is given by
\begin{equation*}
\pi^D(x,y)=x\cdot_{A}y,\quad
 \forall x,y\in A.
\end{equation*}
Moreover, $(A';\rho^D,\mu^D)$ is an associative representation of $(A,\pi^D)$, where $\rho^D$ and $\mu^D$ are given by
\begin{eqnarray*}
\rho^D(x)(u)&=&\rho(x)(u)+D(x)\cdot_{A'}u,\\
\mu^D(x)(u)&=&\mu(x)(u)+u\cdot_{A'}D(x),
\end{eqnarray*}
for all $x \in A, u\in A'$.

The corresponding cohomology is taken to be the {\bf cohomology of a crossed homomorphism} from the associative algebra $(A,\cdot_{A})$ to the associative algebra $(A',\cdot_{A'})$. See \cite{Das3} for more details.
\end{ex}

\begin{ex}
Consider the quasi-twilled associative algebra $(A\oplus A',A,A')$ given in Example \ref{direct product}. Let $D$ be an associative algebra homomorphism from $(A,\cdot_{A})$ to $(A',\cdot_{A'})$.
Then $(A,\pi^D)$ is an associative algebra, where $\pi^D$ is given by
\begin{equation*}
\pi^D(x,y)=x\cdot_{A}y,\quad
 \forall x,y\in A.
\end{equation*}
Moreover, $(A';\rho^D,\mu^D)$ is a representation of $(A,\pi^D)$, where $\rho^D$ and $\mu^D$ are given by
\begin{eqnarray*}
\rho^D(x)(u)&=&D(x)\cdot_{A'}u,\\
\mu^D(x)(u)&=&u\cdot_{A'}D(x),
\end{eqnarray*}
for all $x \in A, u\in A'$.

The corresponding cohomology is taken to be the {\bf cohomology of an associative algebra homomorphism} from $(A,\cdot_{A})$ to $(A',\cdot_{A'})$. See \cite{Barmeier,BD,Fregier,Gerstenhaber3} more details.
\end{ex}

\section{The controlling algebra and cohomology of left deformation maps of quasi-twilled associative algebras}\label{sec:deformation map-2}
%In this section, $(\huaG,\g,\h)$ is always a quasi-twilled associative algebra and the multiplication on $\huaG$ is denoted by
%\begin{equation}
%\nonumber\Omega=\pi+\rho+\mu+\xi+\eta+\beta+\theta,
%\end{equation}
%where $\pi \in \Hom(\g\otimes \g,\g)$, $\rho \in \Hom(\g\otimes \h,\h)$, $\mu \in \Hom(\h\otimes \g,\h)$, $\xi \in \Hom(\g\otimes \h,\g)$, $\eta \in \Hom(\h\otimes \g,\g)$, $\beta \in \Hom(\h\otimes \h,\h)$, $\theta \in \Hom(\g\otimes \g,\h)$.

\subsection{Left deformation maps of quasi-twilled associative algebras}
\
\newline
\indent
In this subsection, we introduce the notion of a left deformation map of a quasi-twilled associative algebra and give some examples, such as a relative Rota-Baxter operator of weight 0 on an associative algebra, a twisted Rota-Baxter operator on an associative algebra, a Reynolds operator on an associative algebra and a deformation map of a matched pair of associative algebras.
\begin{defi}
Let $(\huaA,A,A')$ be a quasi-twilled associative algebra. A {\bf left deformation map} of $(\huaA,A,A')$ is a linear map $B:A'\longrightarrow A$ such that for all $u,v\in A'$,
\begin{equation}
\pi(B(u),B(v))+\xi(B(u),v)+\eta(u,B(v))=B(\beta(u,v)+\rho(B(u),v)+\mu(u,B(v))+\theta(B(u),B(v))).
\end{equation}
\end{defi}
These two types of deformation maps have the following relation:
\begin{pro}
Let $(\huaA,A,A')$ be a quasi-twilled associative algebra and $D:A\longrightarrow A'$ an invertible linear map. Then $D$ is a right deformation map of $(\huaA,A,A')$ if and only if $D^{-1}:A'\longrightarrow A$ is a left deformation map of $(\huaA,A,A')$.
\end{pro}
\begin{proof}
It follows from straightforward computations.
\end{proof}

\begin{ex}(relative Rota-Baxter operator of weight 0)
Consider the quasi-twilled associative algebra $(A\ltimes_{(\rho,\mu)}V,A,V)$ given in Example \ref{semi-direct product}. In this case, a left deformation map of $(A\ltimes_{(\rho,\mu)}V,A,V)$ is a linear map $B:V\longrightarrow A$ such that
\begin{equation}
B(u)\cdot_{A}B(v)=B(\rho(B(u))v+\mu(B(v))u),\quad
 \forall u,v\in V,
\end{equation}
which implies that $B$ is a {\bf relative Rota-Baxter operator of weight 0} (also called $\huaO$-operator) on the associative algebra $(A,\cdot_{A})$ with respect to the representation $(V;\rho,\mu)$. See \cite{Bai1,Bai2,Bai3,Das2,Das4,Uchino} for more details.
\end{ex}

\begin{ex}(twisted Rota-Baxter operator)
Consider the quasi-twilled associative algebra $(A\ltimes_{(\rho,\mu,\omega)}V,A,V)$ given in Example \ref{abelian-extension}. In this case, a left deformation map of $(A\ltimes_{(\rho,\mu,\omega)}V,A,V)$ is a linear map $B:V\longrightarrow A$ such that
\begin{equation}
B(u)\cdot_{A}B(v)=B(\rho(B(u))v+\mu(B(v))u+\omega(B(u),B(v))),\quad
 \forall u,v\in V,
\end{equation}
which implies that $B$ is a {\bf twisted Rota-Baxter operator} on the associative algebra $(A,\cdot_{A})$ with respect to the representation $(V;\rho,\mu)$. See \cite{Das} for more details.
\end{ex}
\begin{ex}(Reynolds operator)
Consider the quasi-twilled associative algebra $(A\ltimes_{(L,R,\omega)}A,A,A)$ given in Example \ref{ex:Reynolds}. In this case, a left deformation map of $(A\ltimes_{(L,R,\omega)}A,A,A)$ is a linear map $B:A\longrightarrow A$ such that
\begin{equation}
B(u)\cdot_{A}B(v)=B(B(u)\cdot_{A}v+u\cdot_{A}B(v)+B(u)\cdot_{A}B(v)),\quad
 \forall u,v\in A,
\end{equation}
which implies that $B$ is a {\bf Reynolds operator} on the associative algebra $(A,\cdot_{A})$. See \cite{Das} for more details.
\end{ex}

\begin{ex}(deformation map)
Consider the quasi-twilled associative algebra $(A\bowtie A',A,A')$ given in Example \ref{matched-pair}. In this case, a left deformation map of $(A\bowtie A',A,A')$ is a linear map $B:A'\longrightarrow A$ such that
\begin{equation}
B(u)\cdot_{A}B(v)+\xi(v)B(u)+\eta(u)B(v)=B(u\cdot_{A'} v+\rho(B(u))v+\mu(B(v))u),\quad
 \forall u,v\in A',
\end{equation}
which implies that $B$ is a {\bf deformation map} of a matched pair $(A,A')$ of associative algebras. See \cite{Agore} for more details.
\end{ex}

\begin{defi}{\rm(\cite{Uchino})}
Let $B:A'\longrightarrow A$ be a linear map. The transformation $\Omega^B\triangleq e^{X_B}(\Omega)$ is called a {\bf twisting} of $\Omega$ by $B$, where $X_B=[\cdot,B]_G$, $X_B^2:=[[\cdot,B]_G,B]_G$ and $X_B^n$ is defined by the same manner.
\end{defi}
\begin{pro}{\rm(\cite{Uchino})}
With the above notation,
\begin{equation*}
\Omega^B=e^{-B}\circ \Omega\circ (e^B\otimes e^B)
\end{equation*}
is an associative algebra structure on $\huaA$ and $e^B$ is an associative algebra isomorphism from $(\huaA,\Omega^B)$ to $(\huaA,\Omega)$.
\end{pro}

Let $(\huaA,A,A')$ be a quasi-twilled associative algebra and $B:A'\longrightarrow A$ a linear map. Write
\begin{equation*}
\Omega^B=\pi^B+\xi^B+\eta^B+\gamma^B+\beta^B+\rho^B+\mu^B+\theta^B,
\end{equation*}
where $\pi^B \in \Hom(A\otimes A,A)$, $\xi^B \in \Hom(A\otimes A',A)$, $\eta^B \in \Hom(A'\otimes A,A)$, $\gamma^B \in \Hom(A'\otimes A',A)$, $\beta^B \in \Hom(A'\otimes A',A')$, $\rho^B \in \Hom(A\otimes A',A')$, $\mu^B \in \Hom(A'\otimes A,A')$, $\theta^B \in \Hom(A\otimes A,A')$. Then we have the following result:
\begin{thm}\label{twisting-ass-deformation-2}
With the above notation, for all $x,y\in A, u,v\in A'$, we have
\begin{eqnarray}
\label{deformation2-1}\pi^B(x,y)&=&\pi(x,y)-B(\theta(x,y)),\\
\label{deformation2-4}\xi^B(x,u)&=&\xi(x,u)+\pi(x,B(u))-B(\rho(x,u))-B(\theta(x,B(u))),\\
\label{deformation2-5}\eta^B(u,x)&=&\eta(u,x)+\pi(B(u),x)-B(\mu(u,x))-B(\theta(B(u),x)),\\
\label{deformation2-8}\gamma^B(u,v)&=&\xi(B(u),v)+\eta(u,B(v))-B(\beta(u,v))+\pi(B(u),B(v))\\
\nonumber&&-B(\rho(B(u),v))-B(\mu(u,B(v)))-B(\theta(B(u),B(v))),\\
\label{deformation2-6}\beta^B(u,v)&=&\beta(u,v)+\rho(B(u),v)+\mu(u,B(v))+\theta(B(u),B(v)),\\
\label{deformation2-2}\rho^B(x,u)&=&\rho(x,u)+\theta(x,B(u)),\\
\label{deformation2-}\mu^B(u,x)&=&\mu(u,x)+\theta(B(u),x),\\
\label{deformation2-7}\theta^B(x,y)&=&\theta(x,y).
\end{eqnarray}
Furthermore, $((\huaA,\Omega^B),A,A')$ is a quasi-twilled associative algebra if and only if $B:A'\longrightarrow A$ is a left deformation map of the quasi-twilled associative algebra $(\huaA, A, A')$.
\end{thm}

\begin{proof}
Similar to Theorem \ref{twisting-ass-deformation-1}, by direct computation, we can obtain \eqref{deformation2-1}-\eqref{deformation2-8}. Furthermore,  $((\huaA,\Omega^B),A,A')$ is a quasi-twilled associative algebra if and only if $\gamma^B=0$, i.e. $B$ is a left deformation map of the quasi-twilled associative algebra $(\huaA, A, A')$.
\end{proof}
\subsection{Controlling algebras of left deformation maps of quasi-twilled associative algebras}

\
\newline
\indent
In this subsection, given a quasi-twilled associative algebra, we obtain a curved $L_\infty$-algebra. Then, we give the controlling algebra of left deformation maps of quasi-twilled associative algebras, which is a curved $L_\infty$-algebra. An important byproduct is the controlling algebra of deformation maps of a matched pair of associative algebras, which is totally unknown before.
\begin{thm}\label{curved-L-1}
Let $(\huaA, A, A')$ be a quasi-twilled associative algebra. Then there is a $V$-data $(L,F,P,\Delta)$ as follows:
\begin{itemize}
\item[$\bullet$] the graded Lie algebra $(L,[\cdot,\cdot])$ is given by $(\oplus_{n=0}^{+\infty}\Hom(\otimes^{n+1}(A\oplus A'), A\oplus A'), [\cdot, \cdot]_G)$,
\item[$\bullet$] the abelian graded Lie subalgebra $F$ is given by $\oplus_{n=0}^{+\infty}\Hom(\otimes^{n+1}A', A)$,
\item[$\bullet$] $P:L\lon L$ is the projection onto the subspace $F$,
\item[$\bullet$] $\Delta=\pi+\xi+\eta+\beta+\rho+\mu+\theta$.
\end{itemize}
Consequently, we obtain an $L_{\infty}$-algebra $(\oplus_{n=0}^{+\infty}\Hom(\otimes^{n+1}A', A), l_1, l_2,l_3)$, where $l_1,l_2,l_3$ are given by
\begin{eqnarray*}
l_1(f)&=&[\xi+\eta+\beta,f]_G\\
l_2(f_1,f_2)&=&[[\pi+\rho+\mu, f_1,]_G,f_2]_G\\
l_3(f_1,f_2,f_3)&=&[[[\theta, f_1]_G, f_2]_G,f_3]_G,
\end{eqnarray*}
for all $f\in \Hom(\otimes^mA',A)$, $f_1\in \Hom(\otimes^{m_1}A',A)$, $f_2\in \Hom(\otimes^{m_2}A',A)$ and $f_3\in \Hom(\otimes^{m_3}A',A)$.

Furthermore, a linear map $B:A'\longrightarrow A$ is a left deformation map of $(\huaA, A, A')$ if and only if $B$ is a Maurer-Cartan element of the above $L_{\infty}$-algebra.
\end{thm}
\begin{proof}
It is obvious  that $F$ is an abelian graded Lie subalgebra of $L$. Since $P$ is the projection onto $F$, we have $P^2=P$. Obviously, $[\Delta,\Delta]_G=0$ and $P(\Delta)=0$. Thus, $(L,F,P,\Delta)$ is a $V$-data. By Theorem \ref{cV}, for all $f\in \Hom(\otimes^mA',A)$, $f_1\in \Hom(\otimes^{m_1}A',A)$, $f_2\in \Hom(\otimes^{m_2}A',A)$ and $f_3\in \Hom(\otimes^{m_3}A',A)$, by direct computation, we have
\begin{eqnarray*}
l_1(f)&=&P([\pi+\xi+\eta+\beta+\rho+\mu+\theta, f]_G)=[\xi+\eta+\beta, f]_G,\\
l_2(f_1,f_2)&=&P([[\pi+\xi+\eta+\beta+\rho+\mu+\theta, f_1]_G,f_2]_G)=[[\pi+\rho+\mu, f_1]_G,f_2]_G,\\
l_3(f_1,f_2,f_3)&=&P([[[\pi+\xi+\eta+\beta+\rho+\mu+\theta, f_1]_G, f_2]_G,f_3]_G)=[[[\theta, f_1]_G, f_2]_G,f_3]_G.
\end{eqnarray*}
By direct computation, we also have
$$
[[[\pi+\xi+\eta+\beta+\rho+\mu+\theta, f_1]_G, f_2]_G,f_3]_G\in\Hom(\otimes^{m_1+m_2+m_3-1}A', A),
$$
Since $F$ is abelian, we have $l_k=0$ for all $k\geq 4$.
Moreover, for all $u,v\in A'$, we have
\begin{eqnarray*}
  &&l_1(B)(u,v)+\frac{1}{2}l_2(B, B)(u,v)+\frac{1}{6}l_3(B,B,B)(u, v)\\
  &=&[\xi+\eta+\beta, B]_G(u,v)+\frac{1}{2}[[\pi+\rho+\mu, B]_G,B]_G(u,v)+\frac{1}{6}[[[\theta, B]_G, B]_G,B]_G(u,v)\\
  &=&\xi(B(u),v)+\eta(u,B(v))-B(\beta(u,v))+\pi(B(u),B(v))\\
  &&-B(\rho(B(u),v))-B(\mu(u,B(v)))-B(\theta(B(u),B(v))).
\end{eqnarray*}
Thus, $B$ is a left deformation map of $(\huaA, A, A')$ if and only if $B$ is a Maurer-Cartan element of the $L_\infty$-algebra $(\oplus_{n=0}^{+\infty}\Hom(\otimes^{n+1}A', A), l_1, l_2,l_3)$.
\end{proof}
\begin{cor}
Consider the quasi-twilled associative algebra $(A\ltimes_{(\rho,\mu)}V,A,V)$ given in Example \ref{semi-direct product}. Then we obtain a graded Lie algebra $(s(\oplus_{n=0}^{+\infty}\Hom(\otimes^{n+1}V, A)),\Courant{\cdot,\cdot})$, where the graded Lie bracket $\Courant{\cdot,\cdot}$ is given by
\begin{eqnarray}
\label{control-semi-product}&&\Courant{f_1,f_2}(u_1,\dots,u_{m_1+m_2})\\
\nonumber &=&(-1)^{m_1-1}l_2(f_1,f_2)(u_1,\dots,u_{m_1+m_2})\\
\nonumber&=&f_2(u_1,\dots,u_{m_2})\cdot_{A}f_1(u_{m_2+1},\dots,u_{m_1+m_2})-(-1)^{m_1m_2}f_1(u_1,\dots,u_{m_1})\cdot_{A}f_2(u_{m_1+1},\dots,u_{m_1+m_2})\\
\nonumber&&+\sum_{i=1}^{m_1}(-1)^{im_2}f_1(u_1,\dots,u_{i-1},\mu(u_i,f_2(u_{i+1},\dots,u_{i+m_2})),u_{i+m_2+1},\dots,u_{m_1+m_2})\\
\nonumber&&-\sum_{i=1}^{m_1}(-1)^{(i-1)m_2}f_1(u_1,\dots,u_{i-1},\rho(f_2(u_i,\dots,u_{i+m_2-1}),u_{i+m_2}),u_{i+m_2+1},\dots,u_{m_1+m_2})\\
\nonumber&&+(-1)^{m_1m_2}\sum_{i=1}^{m_2}(-1)^{(i-1)m_1}f_2(u_1,\dots,u_{i-1},\rho(f_1(u_i,\dots,u_{i+m_1-1}),u_{i+m_1}),u_{i+m_1+1},\dots,u_{m_1+m_2})\\
\nonumber&&-(-1)^{m_1(m_2+1)}\sum_{i=1}^{m_2}(-1)^{(i-1)m_1}f_2(u_1,\dots,u_{i-1},\mu(u_i,f_1(u_{i+1},\dots,u_{i+m_1})),u_{i+m_1+1},\dots,u_{m_1+m_2}),
\end{eqnarray}
for all $f_1\in\Hom(\otimes^{m_1}V, A), f_2\in\Hom(\otimes^{m_2}V, A)$ and $u_1,\dots,u_{m_1+m_2}\in  V$.

Moreover, Maurer-Cartan elements of the graded Lie algebra $(s(\oplus_{n=0}^{+\infty}\Hom(\otimes^{n+1}V, A)),\Courant{\cdot,\cdot})$ are exactly relative Rota-Baxter operators of weight $0$ on the associative algebra $(A,\cdot_{A})$ with respect to the representation $(V;\rho,\mu)$.  Thus, this graded Lie algebra can be viewed as the {\bf controlling algebra of relative Rota-Baxter operators of weight 0} on the associative algebra $(A,\cdot_{A})$ with respect to the representation $(V;\rho,\mu)$. See \cite{Das2,Das4} for more details.
\end{cor}

\begin{cor}\label{control-extension}
Consider the quasi-twilled associative algebra $(A\ltimes_{(\rho,\mu,\omega)}V,A,V)$ given in Example \ref{abelian-extension}. Then $(\oplus_{n=0}^{+\infty}\Hom(\otimes^{n+1}V, A),l_1,l_2,l_3)$ is an $L_\infty$-algebra, where $l_1, l_2, l_3$ are given by $l_1=0$ and
\begin{eqnarray*}
&&l_2(f_1,f_2)(u_1,\dots,u_{m_1+m_2})\\
&=&(-1)^{m_1-1}f_2(u_1,\dots,u_{m_2})\cdot_{A}f_1(u_{m_2+1},\dots,u_{m_1+m_2})\\
&&+(-1)^{m_1(m_2-1)}f_1(u_1,\dots,u_{m_1})\cdot_{A}f_2(u_{m_1+1},\dots,u_{m_1+m_2})\\
&&-(-1)^{m_1}\sum_{i=1}^{m_1}(-1)^{im_2}f_1(u_1,\dots,u_{i-1},\mu(u_i,f_2(u_{i+1},\dots,u_{i+m_2})),u_{i+m_2+1},\dots,u_{m_1+m_2})\\
&&+(-1)^{m_1}\sum_{i=1}^{m_1}(-1)^{(i-1)m_2}f_1(u_1,\dots,u_{i-1},\rho(f_2(u_i,\dots,u_{i+m_2-1}),u_{i+m_2}),u_{i+m_2+1},\dots,u_{m_1+m_2})\\
&&-(-1)^{m_1(m_2-1)}\sum_{i=1}^{m_2}(-1)^{(i-1)m_1}f_2(u_1,\dots,u_{i-1},\rho(f_1(u_i,\dots,u_{i+m_1-1}),u_{i+m_1}),u_{i+m_1+1},\dots,u_{m_1+m_2})\\
&&+(-1)^{m_1m_2}\sum_{i=1}^{m_2}(-1)^{(i-1)m_1}f_2(u_1,\dots,u_{i-1},\mu(u_i,f_1(u_{i+1},\dots,u_{i+m_1})),u_{i+m_1+1},\dots,u_{m_1+m_2}),
\end{eqnarray*}
for all $f_1\in\Hom(\otimes^{m_1}V, A), f_2\in\Hom(\otimes^{m_2}V, A)$ and $u_1,\dots,u_{m_1+m_2}\in  V$.
\begin{eqnarray*}
&&l_3(f_1,f_2,f_3)(u_1,\dots,u_{m_1+m_2+m_3-1})\\
&=&(-1)^{m_1m_2m_3}\big(\sum_{i=1}^{m_1}(-1)^{(i-1)m_2}f_1(u_1,\dots,u_{i-1},\omega(f_2(u_i,\dots,u_{i+m_2-1}),f_3(u_{i+m_2},\dots,u_{i+m_2+m_3-1})),\\
&&\dots,u_{m_1+m_2+m_3-1})\\
&&-(-1)^{m_2m_3}\sum_{i=1}^{m_1}(-1)^{(i-1)m_3}f_1(u_1,\dots,u_{i-1},\omega(f_3(u_i,\dots,u_{i+m_3-1}),f_2(u_{i+m_3},\dots,u_{i+m_2+m_3-1})),\\
&&\dots,u_{m_1+m_2+m_3-1})\\
&&-(-1)^{m_1m_2}\sum_{i=1}^{m_2}(-1)^{(i-1)m_1}f_2(u_1,\dots,u_{i-1},\omega(f_1(u_i,\dots,u_{i+m_1-1}),f_3(u_{i+m_1},\dots,u_{i+m_1+m_3-1})),\\
&&\dots,u_{m_1+m_2+m_3-1})\\
&&+(-1)^{m_1(m_2+m_3)}\sum_{i=1}^{m_2}(-1)^{(i-1)m_3}f_2(u_1,\dots,u_{i-1},\omega(f_3(u_i,\dots,u_{i+m_3-1}),f_1(u_{i+m_3},\dots,u_{i+m_1+m_3-1})),\\
&&\dots,u_{m_1+m_2+m_3-1})\\
&&-(-1)^{m_1m_2+m_2m_3+m_1m_3}\sum_{i=1}^{m_3}(-1)^{(i-1)m_2}f_3(u_1,\dots,u_{i-1},\omega(f_2(u_i,\dots,u_{i+m_2-1}),f_1(u_{i+m_2},\dots,u_{i+m_1+m_2-1})),\\
&&\dots,u_{m_1+m_2+m_3-1})\\
&&+(-1)^{(m_1+m_2)m_3}\sum_{i=1}^{m_3}(-1)^{(i-1)m_1}f_3(u_1,\dots,u_{i-1},\omega(f_1(u_i,\dots,u_{i+m_1-1}),f_2(u_{i+m_1},\dots,u_{i+m_1+m_2-1})),\\
&&\dots,u_{m_1+m_2+m_3-1})\big),
\end{eqnarray*}
for all $f_1\in\Hom(\otimes^{m_1}V, A), f_2\in\Hom(\otimes^{m_2}V, A)$, $f_3\in \Hom(\otimes^{m_3}V,A)$ and $u_1,\dots,u_{m_1+m_2+m_3-1}\in  V$.

Moreover, Maurer-Cartan elements of the $L_\infty$-algebra $(\oplus_{n=0}^{+\infty}\Hom(\otimes^{n+1}V, A),l_2,l_3)$ are exactly twisted Rota-Baxter operators on the associative algebra $(A,\cdot_{A})$ with respect to the representation $(V;\rho,\mu)$.  Thus, this $L_\infty$-algebra can be viewed as the {\bf controlling algebra of twisted Rota-Baxter operators} on the associative algebra $(A,\cdot_{A})$ with respect to the representation $(V;\rho,\mu)$. See \cite{Das} for more details.
\end{cor}

\begin{cor}
Consider the quasi-twilled associative algebra $(A\ltimes_{(L,R,\omega)}A,A,A)$ given in Example \ref{ex:Reynolds}.
Parallel to Corollary \ref{control-extension}, $(\oplus_{n=0}^{+\infty}\Hom(\otimes^{n+1}A, A), l_2, l_3)$ is an $L_{\infty}$-algebra. Moreover, Maurer-Cartan elements of the $L_\infty$-algebra $(\oplus_{n=0}^{+\infty}\Hom(\otimes^{n+1}A, A),l_2,l_3)$ are exactly Reynolds operators on the associative algebra $(A,\cdot_{A})$.  Thus, this $L_\infty$-algebra can be viewed as the {\bf controlling algebra of Reynolds operators} on the associative algebra $(A,\cdot_{A})$. See \cite{Das} for more details.
\end{cor}

\begin{cor}
Consider the quasi-twilled associative algebra $(A\bowtie A',A,A')$ given in Example \ref{matched-pair}. Then we obtain a differential graded Lie algebra $(s(\oplus_{n=0}^{+\infty}\Hom(\otimes^{n+1}A', A)),\Courant{\cdot,\cdot},\dM)$, where the graded Lie bracket $\Courant{\cdot,\cdot}$ is given by \eqref{control-semi-product} and the differential $\dM$ is given by
\begin{eqnarray*}
\dM(u_1,\dots,u_{m+1})&=&\xi(f(u_1,\dots,u_m),u_{m+1})+(-1)^{m-1}\eta(u_1,f(u_2,\dots,u_{m+1}))\\
&&+(-1)^{m-1}\sum_{i=1}^m(-1)^if(u_1,\dots,u_{i-1},\beta(u_i,u_{i+1}),u_{i+2},\dots,u_{m+1})
\end{eqnarray*}
for all $f\in\Hom(\otimes^mA', A)$ and $u_1,\dots,u_{m+1}\in A'$.

Moreover, Maurer-Cartan elements of the differential graded Lie algebra $(s(\oplus_{n=0}^{+\infty}\Hom(\otimes^{n+1}A', A)),\\ \Courant{\cdot,\cdot},\dM)$ are exactly deformation maps of a matched pair $(A,A')$ of associative algebras.  Thus, this differential graded Lie algebra is called the {\bf controlling algebra of deformation maps of a matched pair $(A,A')$ of associative algebras}.
\end{cor}
Let $B:A'\longrightarrow A$ be a left deformation map of a quasi-twilled associative algebra $(\huaA, A, A')$. By Theorem \ref{curved-L-1}, $B$ is a  Maurer-Cartan element of the $L_{\infty}$-algebra  $(\oplus_{n=0}^{+\infty}\Hom(\oplus^{n+1}A', A),\\ l_1, l_2, l_3)$. By Theorem \ref{twistLin}, we have a twisted $L_\infty$-algebra structure on $\oplus_{n=0}^{+\infty}\Hom(\otimes^{n+1}A', A)$ as following:
\begin{eqnarray}
\label{twist-L-3}l_1^{B}(f)&=&l_1(f)+l_2(B,f)+\frac{1}{2}l_3(B,B,f),\\
\label{twist-L-4}l_2^{B}(f_1,f_2)&=&l_2(f_1,f_2)+l_3(B,f_1,f_2),\\
\label{twist-L-5}l_3^{B}(f_1,f_2,f_3)&=&l_3(f_1,f_2,f_3),\\
\label{twist-L-6}l^B_k&=&0,\,\,\,\,k\ge4,
\end{eqnarray}
where $f\in \Hom(\otimes^mA',A)$, $f_1\in \Hom(\otimes^{m_1}A',A)$, $f_2\in \Hom(\otimes^{m_2}A',A)$ and $f_3\in \Hom(\otimes^{m_3}A',A)$.

\begin{thm}
Let $B:A'\longrightarrow A$ be a left deformation map of a quasi-twilled associative algebra $(\huaA, A, A')$ and $B':A'\longrightarrow A$ a linear map. Then $B+B'$ is a left deformation map of $(\huaA, A, A')$ if and only if $B'$ is a Maurer-Cartan element of the twisted $L_\infty$-algebra $(\oplus_{n=0}^{+\infty}\Hom(\otimes^{n+1}A', A), l^B_1, l^B_2,l^B_3)$.
\end{thm}
\begin{proof}
The proof is similar to Theorem \ref{thm:MC-twist-L}, we omit details.
%Let $B+B'$ is a deformation map of $(\huaG, \g, \h)$. Since $B$ is a deformation map of $(\huaG, \g, \h)$, by Theorem \ref{curved-L-1}, we have
%\begin{eqnarray*}
%0&=&l_1(B+B')+\frac{1}{2}l_2(B+B',B+B')+\frac{1}{6}l_3(B+B',B+B',B+B')\\
%&=&l_1(B')+l_2(B,B')+\frac{1}{2}l_2(B',B')+\frac{1}{2}l_3(B,B,B')+\frac{1}{2}l_3(B,B',B')+\frac{1}{6}l_3(B',B',B').
%\end{eqnarray*}
%which implies
%$$
%l_1^{B}(B')+\frac{1}{2}l_2^{B}(B',B')+\frac{1}{6}l_3^{B}(B',B',B')=0,
%$$
%Thus, $B'$ is a Maurer-Cartan element of the twisted $L_\infty$-algebra $(\oplus_{n=0}^{+\infty}\Hom(\otimes^{n+1}\h, \g), l^B_1, l^B_2,l^B_3)$.
%The converse part can be proved similarly. We omit the details. The proof is finished.
\end{proof}

\subsection{Cohomologies of left deformation maps of quasi-twilled associative algebras}
\
\newline
\indent
In this subsection, we introduce the cohomology of left deformation maps of quasi-twilled associative algebras, and recover the cohomology of a relative Rota-Baxter operator of weight 0, a twisted Rota-Baxter operator and a Reynolds operator. In particular, we give the cohomology of a deformation map of a matched pair of associative algebras.
\begin{lem}
Let $B:A'\longrightarrow A$ be a left deformation map of a quasi-twilled associative algebra $(\huaA,A,A')$. Then $(A',\beta^B)$ is an associative algebra
and $(A;\eta^B,\xi^B)$ is a representation of $(A',\beta^B)$, where $\xi^B,\eta^B,\beta^B$ are given by \eqref{deformation2-4},\eqref{deformation2-5},\eqref{deformation2-6} respectively.
\end{lem}
\begin{proof}
By Theorem \ref{twisting-ass-deformation-2}, $((\huaA,\Omega^B),A,A')$ is a quasi-twilled associative algebra, we have $[\Omega^B,\Omega^B]_G=0$. Thus, by \eqref{associative-structure}, we have $[\beta^B,\beta^B]_G=0$, which implies that $\beta^B$ is an associative algebra structure on $A'$. By \eqref{associative-structure}, we also have
\begin{eqnarray}
\label{twisting-cohomology-4}\frac{1}{2}[\xi^B,\xi^B]_G-(\xi^B\circ\beta^B)_2&=&0,\\
\label{twisting-cohomology-5}\frac{1}{2}[\eta^B,\eta^B]_G+(\eta^B\circ\beta^B)_1&=&0,\\
\label{twisting-cohomology-6}[\xi^B,\eta^B]_G&=&0.
\end{eqnarray}
Thus, by \eqref{twisting-cohomology-4}, \eqref{twisting-cohomology-5} and \eqref{twisting-cohomology-6}, we obtain that $(A;\eta^B,\xi^B)$ is a representation of $(A',\beta^B)$.
\end{proof}
Consider the representation  $(A;\eta^B,\xi^B)$  of the associative algebra $(A',\beta^B)$. Define the set of $0$-cochains $C^{0}(B)$ to be $A$. For $n\geq 1$, define the set of $n$-cochains $C^{n}(B)$ by $C^{n}(B)=\Hom(\otimes^nA',A)$. The coboundary operator $\dM^B:C^n(B)\longrightarrow C^{n+1}(B)$ is defined by
\begin{eqnarray*}
\dM^B f(u_1,\dots,u_{n+1})&=&\eta^B(u_1)f(u_2,\dots,x_{n+1})+\sum_{i=1}^n(-1)^if(u_1,\dots,u_{i-1},\beta^B(u_i, u_{i+1}),u_{i+2},\dots,u_{n+1})\\
&&+(-1)^{n+1}\xi^B(u_{n+1})f(u_1,\dots,u_n)\\
&=&\eta(u_1,f(u_2,\dots,u_{n+1}))+\pi(B(u_1),f(u_2,\dots,u_{n+1}))\\
&&-B(\mu(u_1,f(u_2,\dots,u_{n+1})))-B(\theta(B(u_1),f(u_2,\dots,u_{n+1})))\\
&&+\sum_{i=1}^n(-1)^if(u_1,\dots,u_{i-1},\beta(u_i, u_{i+1}),u_{i+2},\dots,u_{n+1})\\
&&+\sum_{i=1}^n(-1)^if(u_1,\dots,u_{i-1},\rho(B(u_i),u_{i+1}),u_{i+2},\dots,u_{n+1})\\
&&+\sum_{i=1}^n(-1)^if(u_1,\dots,u_{i-1},\mu(u_i, B(u_{i+1})),u_{i+2},\dots,u_{n+1})\\
&&+\sum_{i=1}^n(-1)^if(u_1,\dots,u_{i-1},\theta(B(u_i), B(u_{i+1})),u_{i+2},\dots,u_{n+1})\\
&&+(-1)^{n+1}\xi(f(u_1,\dots,u_n),u_{n+1})+(-1)^{n+1}\pi(f(u_1,\dots,u_n),B(u_{n+1}))\\
&&-(-1)^{n+1}B(\rho(f(u_1,\dots,u_n),u_{n+1}))-(-1)^{n+1}B(\theta(f(u_1,\dots,u_n),B(u_{n+1}))),
\end{eqnarray*}
for all $f\in \Hom(\otimes^nA',A)$ and $u_1,\dots,u_{n+1}\in A'$.
\begin{defi}
The cohomology of the cochain complex $(\oplus_{i=0}^{+\infty}C^i(B),\dM^B)$ is called the {\bf cohomology of left deformation maps} of the quasi-twilled associative algebra $(\huaA,A,A')$.
 \end{defi}
Denote the set of closed $n$-cochains by $Z^n(B)$ and the set of exact $n$-cochains by $B^n(B)$. We denote by $H^n(B)$=$Z^n(B)$/$B^n(B)$ the corresponding cohomology group.

\begin{pro}
Let $B:A'\longrightarrow A$ be a left deformation map of a quasi-twilled associative algebra $(\huaA,A,A')$ and $(\oplus_{n=0}^{+\infty}\Hom(\otimes^{n+1}A', A), l^B_1, l^B_2,l^B_3)$ a twisted $L_\infty$-algebra. Then for all $f\in \Hom(\otimes^mA', A)$, we have
$$l_1^B(f)=(-1)^{m-1}\dM^{B}f.$$
\end{pro}

\begin{proof}
The proof is similar to Proposition \ref{cohomology-deformation-map}, we omit details.
\end{proof}
\begin{ex}
Consider the quasi-twilled associative algebra $(A\ltimes_{(\rho,\mu)}V,A,V)$ given in Example \ref{semi-direct product}. Let $B:V\longrightarrow A$ be a relative Rota-Baxter operator of weight $0$ ($\huaO$-operator) on the associative algebra $(A,\cdot_{A})$ with respect to the representation $(V;\rho,\mu)$. Then $(V,\beta^B)$ is an associative algebra, where $\beta^B$ is given by
\begin{equation*}
\beta^B(u,v)=\rho(B(u))(v)+\mu(B(v))(u),\quad
 \forall u,v\in V.
\end{equation*}
Moreover, $(A;\eta^B,\xi^B)$ is a representation of $(V,\beta^B)$, where $\eta^B$ and $\xi^B$ are given by
\begin{eqnarray*}
\eta^B(u)(x)&=&B(u)\cdot_{A}x-B(\mu(x)(u)),\\
\xi^B(u)(x)&=&x\cdot_{A}B(u)-B(\rho(x)(u)),
\end{eqnarray*}
for all $x \in A, u\in V$.

The corresponding cohomology is taken to be the {\bf cohomology of relative Rota-Baxter operators of weight 0} on the associative $(A,\cdot_{A})$ with respect to the representation $(V;\rho,\mu)$. See \cite{Das2,Das4} for more details.
\end{ex}

\begin{ex}
Consider the quasi-twilled associative algebra $(A\ltimes_{(\rho,\mu,\omega)}V,A,V)$ given in Example \ref{abelian-extension}. Let $B:V\longrightarrow A$ be a twisted Rota-Baxter operator on the associative algebra $(A,\cdot_{A})$ with respect to the representation $(V;\rho,\mu)$. Then $(V,\beta^B)$ is an associative algebra, where $\beta^B$ is given by
\begin{equation*}
\beta^B(u,v)=\rho(B(u))(v)+\mu(B(v))(u)+\omega(B(u),B(v)),\quad
 \forall u,v\in V.
\end{equation*}
Moreover, $(A;\eta^B,\xi^B)$ is a representation of $(V,\beta^B)$, where $\eta^B$ and $\xi^B$ are given by
\begin{eqnarray*}
\eta^B(u)(x)&=&B(u)\cdot_{A}x-B(\mu(x)(u))-B(\omega(B(u),x)),\\
\xi^B(u)(x)&=&x\cdot_{A}B(u)-B(\rho(x)(u))-B(\omega(x,B(u))),
\end{eqnarray*}
for all $x \in A, u\in V$.

The corresponding cohomology is taken to be the {\bf cohomology of twisted Rota-Baxter operators} on the associative algebra $(A,\cdot_{A})$ with respect to the representation $(V;\rho,\mu)$. See \cite{Das} more details.
\end{ex}

\begin{ex}
Consider the quasi-twilled associative algebra $(A\ltimes_{(L,R,\omega)}A,A,A)$ given in Example \ref{ex:Reynolds}. Let $B:A\longrightarrow A$ be a Reynolds operator on the associative algebra $(A,\cdot_{A})$. Then $(A,\beta^B)$ is an associative algebra, where $\beta^B$ is given by
\begin{equation*}
\beta^B(x,y)=B(x)\cdot_{A}y+x\cdot_{A} B(y)+B(x)\cdot_{A}B(y),\quad
 \forall x,y\in A.
\end{equation*}
Moreover, $(A;\eta^B,\xi^B)$ is a representation of $(A,\beta^B)$, where $\eta^B$ and $\xi^B$ are given by
\begin{eqnarray*}
\eta^B(x)(y)&=&B(x)\cdot_{A}y-B(x\cdot_{A} y)-B(B(x)\cdot_{A} y),\\
\xi^B(x)(y)&=&y\cdot_{A}B(x)-B(y \cdot_{A} x)-B(y\cdot_{A}B(x)),
\end{eqnarray*}
for all $x,y\in A$.

The corresponding cohomology is taken to be the {\bf cohomology of Reynolds operators} on the associative algebra $(A,\cdot_{A})$. See \cite{Das} for more details.
\end{ex}

Consider the quasi-twilled associative algebra $(A\bowtie A',A,A')$ given in Example \ref{matched-pair}. Let $B:A'\longrightarrow A$ be a deformation map of a matched pair $(A,A')$ of associative algebras. Then $(A',\beta^B)$ is an associative algebra, where $\beta^B$ is given by
\begin{equation*}
\beta^B(u,v)=u\cdot_{A'}v+\rho(B(u))(v)+\mu(B(v))(u),\quad
 \forall u,v\in A'.
\end{equation*}
Moreover, $(A;\eta^B,\xi^B)$ is a representation of $(A',\beta^B)$, where $\eta^B$ and $\xi^B$ are given by
\begin{eqnarray*}
\eta^B(u)(x)&=&\eta(u)(x)+B(u)\cdot_{A}x-B(\mu(x)(u)),\\
\xi^B(u)(x)&=&\xi(u)(x)+x\cdot_{A}B(u)-B(\rho(x)(u)),
\end{eqnarray*}
for all $x \in A, u\in A'$.
\begin{defi}
The corresponding cohomology is taken to be the {\bf cohomology of deformation maps of a matched pair $(A,A')$ of associative algebras}.

\end{defi}

This cohomology is new and can be used to classify deformation and extension problems, and we leave it to interesting readers.


\begin{thebibliography}{999}
\bibitem{Agore}
A. L. Agore, Classifying complements for associative algebras. \emph{Linear Algebra Appl.} 446 (2014), 345-355.

\bibitem{Bai1}
C. Bai, L. Guo and X. Ni, O-operators on associative algebras and associative Yang-Baxter equations. \emph{Pacific J. Math.} 256 (2012), no. 2, 257-289.

\bibitem{Bai2}
C. Bai, L. Guo and X. Ni, O-operators on associative algebras, associative Yang-Baxter equations and dendriform algebras. \emph{Nankai Ser. Pure Appl. Math. Theoret. Phys.} 8 (2012), 10-51.

\bibitem{Bai3}
C. Bai, L. Guo and X. Ni, Relative Rota-Baxter operators and tridendriform algebras. \emph{J. Algebra Appl.} 12 (2013), no. 7, 1350027, 18 pp.

\bibitem{Bal} D. Balavoine, Deformations of algebras over a quadratic operad. Operads: Proc. of Renaissance Conferences (Hartford, CT/Luminy, 1995), \emph{Contemp. Math.} {\em 202} Amer. Math. Soc., Providence, RI, 1997, 207-34.

\bibitem{Barmeier}
S. Barmeier and Y. Fr\'egier, Deformation-obstruction theory for diagrams of algebras and applications to geometry. \emph{J. Noncommut. Geom.} 14 (2020), 1019-1047.

\bibitem{BD}
D. V. Borisov, Formal deformations of morphisms of associative algebras. \emph{Int. Math. Res. Not.} 2005, no. 41, 2499-2523.

\bibitem{Das2}
A. Das, Deformations of associative Rota-Baxter operators. \emph{J. Algebra} 560 (2020), 144-180.

\bibitem{Das}
A. Das, Cohomology and deformations of twisted Rota-Baxter operators and NS-algebras. \emph{J. Homotopy Relat. Struct.} 17 (2022), no. 2, 233-262.

\bibitem{Das3}
A. Das, Cohomology and deformations of crossed homomorphisms.
\emph{Bull. Belg. Math. Soc. Simon Stevin} 28 (2022), no. 3, 381-397.

\bibitem{Das1}
A. Das, A cohomology study of modified Rota-Baxter algebras. arXiv:2207.02273v1.

\bibitem{Das4}
A. Das, Cohomology and deformations of weighted Rota-Baxter operators. \emph{J. Math. Phys.} 63 (2022), no.9, Paper No. 091703, 16 pp.

\bibitem{DSV}
V. Dotsenko, S. Shadrin and B. Vallette, Maurer-Cartan methods in deformation theory-the twisting procedure. \emph{London Math. Soc.} Lecture Note Ser., 488
Cambridge University Press, Cambridge, 2024, viii+177 pp.

\bibitem{Fregier}
Y. Fr\'egier, M. Markl and D. Yau,   The
$L_\infty$-deformation complex of diagrams of algebras. \emph{New York J. Math.} 15 (2009), 353-392

\bibitem{Get}
E. Getzler, Lie theory for nilpotent $L_{\infty}$-algebras. \emph{Ann. of Math.} (2) 170 (2009), 271-301.

\bibitem{Guo}
L. Guo, Y. Li, Y. Sheng and G. Zhou, Cohomologies, extensions and deformations of differential algebras with any weights. arXiv:2003.03899v1.

\bibitem{Gerstenhaber1} M. Gerstenhaber, The cohomology structure of an associative ring.  \emph{Ann. Math.} 78 (1963), 267-288.

\bibitem{Gerstenhaber2} M. Gerstenhaber, On the deformation of rings and algebras. \emph{Ann. Math.} 79 (1964), 59-103.

\bibitem{Gerstenhaber3} M. Gerstenhaber and S. D. Schack, On the cohomology of an algebra morphism. \emph{J. Algebra} 95 (1985), 245-262.

\bibitem{Hochschild} G. Hochschild, On the cohomology groups of an associative algebra. \emph{Ann. of Math.} (2) 46 (1945), 58-67.

\bibitem{KS}
H. Kajiura and J. Stasheff, Homotopy algebras inspired by classical open-closed string field theory. \emph{Comm. Math. Phys.} 263 (2006), 553-581.

\bibitem{KSo}
M. Kontsevich and Y. Soibelman, Deformation theory. I [Draft], http://www.math.ksu.edu/~soibel/Book-vol1.ps, 2010.

\bibitem{Kosmann-Schwarzbach}
 Y. Kosmann-Schwarzbach, From Poisson algebras to Gerstenhaber algebras. {\em Ann. Inst. Fourier (Grenoble)} 46 (1996), 1243-1274.

\bibitem{LJF-1} J. Liu, C. Bai and Y. Sheng, Compatible O-operators on bimodules over associative algebras. \emph{J. Algebra}  532 (2019), 80-118.

\bibitem{LV}
J. L. Loday and B. Vallette, Algebraic Operads. Springer, 2012.

\bibitem{Loday}
J. L. Loday, On the operad of associative algebras with derivation.  \emph{Georgian Math. J.} 17 (2010), no.2, 347-372.

\bibitem{LS}
T. Lada and J. Stasheff, Introduction to sh Lie algebras for physicists. \emph{Internat. J. Theoret. Phys.}  32 (1993), 1087-1103.

\bibitem{Ma}
M. Markl, Deformation Theory of Algebras and Their Diagrams. \emph{Regional Conference Series in Mathematics}, Number 116, American Mathematical Society (2011).

\bibitem{Ma-0}
M. Markl, Intrinsic brackets and the $L_{\infty}$-deformation theory of bialgebras. \emph{J. Homotopy Relat. Struct.} 5 (2010), 177-212.

\bibitem{NR}
A. Nijenhuis  and R. Richardson,  Cohomology and deformations in graded Lie algebras. {\em Bull. Amer. Math. Soc.} 72 (1966), 1-29.

\bibitem{NR2}
A. Nijenhuis and R. Richardson,  Commutative  algebra cohomology and deformations of Lie and associative algebras. {\em J. Algebra} 9 (1968), 42-105.

\bibitem{Uchino} K. Uchino, Twisting on associative algebras and Rota-Baxter type operators. \emph{J. Noncommut. Geom.} 4 (2010), 349-379.

%\bibitem{Uchino1}K. Uchino, Quantum analogy of Poisson geometry, related dendriform algebras and Rota-Baxter operators. \emph{Lett. Math. Phys.} 85 (2008),
%no. 2-3, 91-109.

\bibitem{Vo}
T. T. Voronov, Higher derived brackets and homotopy algebras. \emph{J. Pure Appl. Algebra}  202 (2005), 133-153.

\bibitem{Guo1} X. Zhang, X. Gao and L. Guo, Modified Rota-Baxter Algebras, Shuffle Products and Hopf Algebras. \emph{Math. Sci. Soc.} 42 (2019), 3047-3072.

\bibitem{Guo2} X. Zhang, X. Gao and L. Guo, Free modified Rota-Baxter algebras and Hopf algebras. \emph{J. Algebra} 25 (2019), 12-34.

\end{thebibliography}
 \end{document}